%% file: main.tex
\documentclass[hidelinks,onefignum,onetabnum]{siamart220329}

\input{ex_shared}

\ifpdf
\hypersetup{
  pdftitle={Second-order robust parallel integrators for dynamical low-rank approximation},
  pdfauthor={Jonas Kusch}
}
\fi

\begin{document}

\maketitle

% REQUIRED
\begin{abstract}
Due to its reduced memory and computational demands, dynamical low-rank approximation (DLRA) has sparked significant interest in multiple research communities. A central challenge in DLRA is the development of time integrators that are robust to the curvature of the manifold of low-rank matrices. Recently, a parallel robust time integrator that permits dynamic rank adaptation and enables a fully parallel update of all low-rank factors was introduced. Despite its favorable computational efficiency, the construction as a first-order approximation to the augmented basis-update \& Galerkin integrator restricts the parallel integrator's accuracy to order one. In this work, an extension to higher order is proposed by a careful basis augmentation before solving the matrix differential equations of the factorized solution. A robust error bound with an improved dependence on normal components of the vector field together with a norm preservation property up to small terms is derived. These analytic results are complemented and demonstrated through a series of numerical experiments.
\end{abstract}

% REQUIRED
\begin{keywords}
dynamical low-rank approximation, matrix differential equations
\end{keywords}

% REQUIRED
\begin{MSCcodes}
65L05, 65L20, 65L70
\end{MSCcodes}

\section{Introduction}\label{sec:intro}
Though dynamical low-rank approximation (DLRA) \cite{KochLubich07} offers a significant reduction of computational costs and memory consumption when solving tensor differential equations~\cite{haegeman2016unifying,ceruti2023rank,sulz2024numerical}, the use of DLRA to solve matrix differential equations has sparked immense interest in several communities. Research fields in which DLRA for matrix differential equations has a considerable impact include plasma physics \cite{EiL18, Einkemmer2020,Einkemmer2023,cassini2022efficient, EiJ21,Coughlin2022,Einkemmer2022,Coughlin2023,Einkemmer2023b,Uschmajew2023}, radiation transport \cite{PeMF20,ding2021dynamical,peng2021high,kusch2022low,peng2023sweep,yin2023semi,einkemmer2024asymptotic,baumann2023energy,kusch2023robust}, chemical kinetics \cite{Jahnke2008,Prugger2023,Einkemmer2023a}, wave propagation \cite{Hochbruck2023,zhao2023low}, uncertainty quantification \cite{SaL09,babaee2017robust,FeL18,MuN18,MuNV20,patil2020real,kusch2021DLRUQ,kazashi2021existence,DoPNH23,ali2024dynamicallyAppl}, and machine learning \cite{schotthofer2022low,zangrando2023rank,savostianova2023robust,schmidt2023rank}. These application fields commonly require memory-intensive and computational costly numerical simulations due to the solution's prohibitively large phase space. The main idea of DLRA to reduce computational costs and memory consumption is to approximate and evolve the solution by a low-rank matrix factorization. A central challenge in computing the time evolution of this low-rank factorization is to construct time integrators that do not exhibit dependence on the curvature of the manifold of low-rank matrices. This curvature becomes prohibitively large when the smallest singular value of the low-rank factorization is close to zero. In this case, standard time integration methods require prohibitively small time step sizes to ensure stability. \\

Low-rank integrators that do not depend on the manifold's curvature have been designed to move only along flat subspaces on the manifold, thereby inheriting the stability region of the original, full-rank problem. Existing robust integrators are the projector--splitting (PS) integrator~\cite{LubichOseledets,KieriLubichWalach}, basis-update \& Galerkin (BUG) integrators \cite{CeL22,CeKL22,CeKL23,ceruti2024robust}, and projection methods \cite{KiV19,charous2023dynamically,SeCK23,carrel2023projected}. While the projector--splitting integrator has proven to yield accurate solutions for various problems, the fact that it includes a substep that evolves the solution backward in time can lead to instabilities not only for parabolic problems but also for numerical stabilization terms in hyperbolic problems \cite{kusch2021stability} and problems exhibiting strong scattering \cite{einkemmer2024asymptotic}. While the projector--splitting integrator can be combined with high-order splitting methods, analytic error bounds only exist for order one \cite{KieriLubichWalach}. 
The fixed-rank BUG integrator \cite{CeL22}, also known as the ``unconventional integrator'' as it does not rely on a splitting scheme, evolves the evolution equations of the basis in parallel, followed by a sequential update of the coefficient matrix. Besides its parallelism in updating the basis matrices, the main advantage of this integrator is that it evolves the solution only forward in time. This integrator's main drawback is its lack of structure--preservation and sometimes poor accuracy, see, e.g., \cite{Einkemmer2023}. A conservative and rank-adaptive extension has been proposed in \cite{CeKL22}, which, compared to \cite{CeL22}, can be shown to preserve local conservation laws for kinetic problems \cite{Einkemmer2022, Einkemmer2023b}. Moreover, it preserves, up to a user-determined tolerance parameter, the norm when the full-rank differential equation does, it preserves the energy for Hamiltonian systems and Schrödinger equations, and it preserves the monotonic decrease of the functional in gradient flows. This integrator, also known as the \emph{augmented BUG} or \emph{rank-adaptive BUG} integrator, comes at the cost of requiring a rank $2r$ coefficient update where $r$ is the rank of the solution at the previous time step. These increased costs have been mitigated by the parallel integrator \cite{CeKL23}, which does not require a potentially expensive coefficient update at rank $2r$. A primary advantage of the parallel integrator is its ability to solve the differential equations of all low-rank factors in parallel. These benefits are achieved by approximating serial terms in the coefficient update as zero, limiting the parallel integrator to order one. Further disadvantages are the lack of the desirable structure preservation properties of the augmented BUG integrator. 

While BUG integrators have shown accurate solution approximations in various fields, the previously mentioned integrators are limited to order one. However, it must be noted that the augmented BUG integrator exhibits a second-order error bound in various numerical experiments \cite[Section~5.2]{ceruti2024robust}, a behavior that is not well understood on the analytic level. Recently, a midpoint BUG integrator has been presented in \cite{ceruti2024robust} that allows for a robust error bound of order two while inheriting the structure--preserving properties of \cite{CeKL22}. The central idea of this integrator is to use the augmentation step proposed in \cite{CeKL22} to enlarge the basis for the Galerkin step with second-order information. This comes at the cost of solving a Galerkin step of rank $3r$ (or $4r$ to achieve an improved error bound), leading to increased computational costs. Moreover, the midpoint BUG integrator requires three sequential steps: two parallel basis updates followed by two sequential Galerkin steps.

This work proposes extensions to the parallel integrator, which achieve second-order convergence in time, independent of the manifold's curvature, and with differential equations for low-rank factors of only rank $2r$. The main novelties are:
\begin{itemize}
    \item \emph{Robust extensions of the parallel integrator to second order}: The resulting integrators again solve all differential equations in parallel; however, they require an augmentation that yields rank $2r$ updates of the low-rank factors. By construction, these integrators are rank-adaptive and evolve the solution forward in time.
    \item \emph{A second order robust error bound}: A robust error bound of order two is shown to hold independent of the low-rank manifold's curvature under standard assumptions \cite{KieriLubichWalach} with improved dependence on normal components, similar to~\cite{ceruti2024robust}.
    \item \emph{Near norm preservation}: Norm preservation is shown to hold up to the truncation tolerance, normal components, and $O(h^4)$ terms, where $h$ is the time step size, and constants are independent of the low-rank manifold's curvature.
\end{itemize}
These results are complemented with numerical experiments that demonstrate the novel integrators' robust second-order error bound and its reduced computational costs. However, it must be noted that the midpoint BUG integrator, though requiring additional operations, yields an improved error for Schr\"odinger equations.

The paper is structured as follows: After the introduction in Section~\ref{sec:intro}, a summary of dynamical low-rank approximation with a particular focus on the parallel BUG integrator is provided in Section~\ref{sec:background}. Section~\ref{sec:2ndorderpar} introduces the two variants of second--order parallel integrators together with a second--order robust error bound provided in Section~\ref{sec:robusterror}. In Section~\ref{sec:preservation}, the norm preservation property of the integrators is studied. Section~\ref{sec:numExp} presents a series of numerical experiments showcasing the parallel integrators' properties.

\section{Background}\label{sec:background}
\subsection{Dynamical low-rank approximation}
In the following, we review fundamental concepts of dynamical low-rank approximation and robust time integration for DLRA.
We consider a (prohibitively large) matrix ordinary differential equation (ODE) of the form
\begin{align}\label{eq:exact}
	\dot \bfA(t) = \bfF(t, \bfA(t)), \quad \bfA(t=0) = \bfA_0\,,
\end{align}
where $\bfA(t)\in\mathbb{R}^{m\times n}$. To efficiently solve problems of this type, dynamical low-rank approximation represents the solution as a low-rank factorization $\bfA(t) \approx \bfX(t):= \bfU(t)\bfS(t)\bfV(t)^{\top} \in \mathcal{M} \subset \mathbb{R}^{m\times n}$ where $\bfU(t)\in\mathbb{R}^{m\times r}$, $\bfV(t)\in\mathbb{R}^{n\times r}$ are basis matrices with orthonormal columns and $\bfS(t)\in\mathbb{R}^{r\times r}$ is the coefficient matrix. To ensure an improved memory footprint, the rank $r$ is chosen such that $r\ll m,n$. Following \cite{KochLubich07}, we denote the manifold of rank $r$ matrices by $\mathcal{M}$ and its tangent space at position $\bfZ$ as $T_{\bfZ}\mathcal{M}$.
Then $\bfX(t)$ is evolved in time according to
\begin{align*}
\dot \bfX(t)\in T_{ \bfX(t)}\mathcal{M} \qquad \text{such that} \qquad \left\Vert \dot\bfX(t)-\mathbf{F}(t, \bfX(t)) \right\Vert = \text{min}\,,
\end{align*}
where $\Vert \cdot \Vert = \Vert \cdot \Vert_F$ denotes the Frobenius norm. That is, the time derivative of $\bfX(t)$ is chosen to lie in the tangent space while minimizing the defect. This formulation can be written as \cite[Lemma~4.1]{KochLubich07} 
\begin{align}\label{eq:M-exact}
\dot \bfX(t) = \bfP(\bfX(t))\bfF(t, \bfX(t)) \,,
\end{align}
where for $\bfX = \bfU\bfS\bfV^{\top}$, the orthogonal projection of $\bfZ$ onto the tangent space at $\bfX$ is given by
\begin{align*}
\bfP(\bfX)\bfZ := \bfU\bfU^{\top}\bfZ(\bfI - \bfV\bfV^{\top}) + \bfZ \bfV\bfV^{\top}\,.
\end{align*}
In the following, we often write 
\begin{align*}
\bfF(t,\bfX) = \bfP(\bfX)\bfF(t,\bfX) + (\bfI - \bfP(\bfX))\bfF(t,\bfX) =: \bfM(t,\bfX) + \bfR(t,\bfX)\,,
\end{align*}
where $\bfM(t,\bfX) := \bfP(\bfX)\bfF(t,\bfX)\in T_{ \bfX}\mathcal{M}$ is the tangent component of the full-rank dynamics and $\bfR(t,\bfX) := (\bfI - \bfP(\bfX))\bfF(t,\bfX)$ is the distance between the projected dynamics of \eqref{eq:M-exact} to the full-rank dynamics of \eqref{eq:exact}. Then, the projected problem \eqref{eq:M-exact} can be rewritten as evolution equations for the individual low-rank factors of $\bfX(t)$ as \cite[Proposition~2.1]{KochLubich07}
\begin{subequations}\label{eq:factorizedDyn}
\begin{align}
\dot \bfU(t) =\,& \bfF(t,\bfX(t))\bfV(t)\bfS(t)^{-1}\,, \\
\dot \bfS(t) =\,& \bfU(t)^{\top}\bfF(t,\bfX(t))\bfV(t)\,, \\
\dot \bfV(t) =\,& \bfF(t,\bfX(t))^{\top}\bfU(t)\bfS(t)^{-\top}\,.
\end{align}
\end{subequations}
Solving \eqref{eq:factorizedDyn} or equivalently \eqref{eq:M-exact} with standard time integration schemes requires prohibitively small time step sizes. This lies in the fact that the projector $\bfP$ is non-smooth \cite[Lemma~4.2]{KochLubich07}, that is, if the smallest non-zero singular value satisfies $\sigma_r(\bfX) \geq \rho > 0$ and $\bfZ\in\mathcal{M}$ with $\Vert \bfZ - \bfX \Vert \leq \frac18 \rho$, then with the spectral norm $\Vert \cdot\Vert_2$,
\begin{align*}
	\Vert (\bfP(\bfZ) - \bfP(\bfX))\bfB \Vert \leq 8\rho^{-1} \Vert \bfZ - \bfX \Vert \cdot \Vert \bfB\Vert_2\,.
\end{align*}
Note that $\rho^{-1}$, i.e., the inverse of the smallest singular value of the solution, can become arbitrarily big, especially if the rank is chosen sufficiently large. From a geometric standpoint, $\rho^{-1}$ can be understood as the curvature of the low-rank manifold $\mathcal{M}$.
A time integrator that does not depend on the potentially large $\rho^{-1}$ term is the projector--splitting integrator \cite{LubichOseledets}, which splits the dynamics into three smooth substeps, thereby conserving the time step restriction of the full problem. A rigorous error analysis showing first order accuracy of the projector--splitting integrator has been presented in \cite{KieriLubichWalach}. The BUG integrators \cite{CeL22,CeKL22,CeKL23} inherit this first-order robust error bound. The midpoint BUG integrator \cite{ceruti2024robust} possesses a second-order robust error bound; however, it requires three sequential updates and a maximal rank of $3r$ (or even $4r$ to ensure an improved dependence on normal components). This work proposes robust extensions of the parallel integrator to order two. The resulting integrators solve all differential equations in parallel with a maximal rank of $2r$. For this, let us first review the parallel integrator in the next section.

\subsection{The parallel basis-update \& Galerkin integrator}

The parallel BUG integrator of \cite{CeKL23} is constructed as a first-order approximation to the augmented BUG integrator of \cite{CeKL22}. To update the rank $r$ solution $\bfY_0 = \bfU_0\bfS_0\bfV_0^{\top}$ from time $t_0$ to time $t_1$, the augmented BUG integrator first updates the basis along flat subspaces of $\mathcal{M}$ according to the $m\times r$ and $n \times r$ matrix ordinary differential equations
\begin{alignat*}{2}
\dot{\bfK}(t) =\,& \bfF(t,\bfK(t)\bfV_0^{\top})\bfV_0\, , \quad &&\bfK(t_0) = \bfU_0 \bfS_0 \, , \\
\dot{\bfL}(t) =\,& \bfF(t,\bfU_0 \bfL(t)^{\top})^{\top}\bfU_0\, , \quad &&\bfL(t_0) = \bfV_0 \bfS_0^{\top}\, ,
\end{alignat*}
in parallel. The orthonormal basis $\wh \bfU\in\mathbf{R}^{m\times 2r}$ is then constructed as the span of $\bfU_0$ and $\bfK(t_1)$, while the orthonormal basis of the co-range $\wh \bfV\in\mathbf{R}^{n\times 2r}$ is constructed as the span of $\bfV_0$ and $\bfL(t_1)$. In the following, this augmented basis is split into $\wh \bfU = [\bfU_0, \widetilde \bfU_1]$ and $\wh \bfV = [\bfV_0, \widetilde \bfV_1]$, that is, the basis consists of the old basis at time $t_0$ and the orthonormal new information gathered when solving the differential equations for $\bfK$ and $\bfL$, denoted by a tilde. 

With $\widehat \bfF(t)=\bfF(t,\wh\bfU \wh\bfS(t) \wh \bfV^\top)$, the augmented BUG integrator then performs a sequential Galerkin step to update the coefficient matrix according to
\begin{align}\label{eq:augmentedSMatrix}
    \dot{\widehat \bfS}(t) = \widehat \bfU^\top \bfF(t,\widehat \bfU \widehat \bfS(t) \widehat \bfV^\top) \widehat \bfV =
\begin{pmatrix} 
\bfU_0^\top \widehat \bfF(t) \bfV_0 & \bfU_0^\top \widehat \bfF(t) \widetilde \bfV_1
\\
\widetilde \bfU_1^\top \widehat \bfF(t) \bfV_0 & \widetilde \bfU_1^\top \widehat \bfF(t) \widetilde \bfV_1
\end{pmatrix}
\end{align}
leading to the rank $2r$ solution $\wh\bfY_{\mathrm{BUG}} = \wh\bfU\widehat \bfS(t_1)\wh\bfV^{\top}$. The augmented BUG integrator truncates the solution to a new rank $r_1<2r$. This yields a robust error bound of order one \cite[Theorem~2]{CeKL22}, that is, all constants in the error bound are independent of $\rho^{-1}$.

The parallel integrator of \cite{CeKL23} updates the augmented bases $\wh \bfU$ and $\wh \bfV$ similarly to the augmented BUG integrator. Instead of removing information gathered when updating the bases, it uses this information to construct a first-order approximation to the coefficient matrix of the augmented BUG integrator. 

In the following algorithm, we denote an orthonormalization algorithm (such as Gram-Schmidt) as $\text{orth}(\bfZ)$. Then, starting at a rank $r$ factorization $\bfY_0 = \bfU_0\bfS_0\bfV_0^{\top}$, the parallel integrator updates the solution factors to a new time $t_1 = t_0 + h$ according to the following algorithm:
\begin{enumerate}
\item Determine the augmented basis matrices $\widehat{\mathbf{U}}_1\in \mathbb{R}^{m\times 2r}$ and $\widehat{\mathbf{V}}_1\in \mathbb{R}^{n\times 2r}$ as well as the coefficient matrix $\bar{\mathbf{S}}(t_1) \in\mathbb{R}^{r \times r}$ (in parallel):
\\[2mm]
\textbf{K-step}: Integrate from $t=t_0$ to $t_1$ the $m \times r$ matrix differential equation
\begin{align*}
\dot{\bfK}(t) =\,& \bfF(t,\bfK(t)\bfV_0^{\top})\bfV_0\, , \quad \bfK(t_0) = \bfU_0 \bfS_0 \, .
\end{align*}
Construct $\widehat \bfU_1= [\bfU_0, \widetilde \bfU_1] = \text{orth}([\bfU_0, \bfK(t_1)])$. 
\\[2mm]
\textbf{L-step}: Integrate from $t=t_0$ to $t_1$ the $n \times r$ matrix differential equation
\begin{align*}
\dot{\bfL}(t) =\,& \bfF(t,\bfU_0 \bfL(t)^{\top})^{\top}\bfU_0\, , \quad \bfL(t_0) = \bfV_0 \bfS_0^{\top}\, .
\end{align*}
Construct $\widehat \bfV_1= [\bfV_0, \widetilde \bfV_1] = \text{orth}([\bfV_0, \bfL(t_1)])$. 
\\[2mm]
\textbf{S-step}: Solve the $r \times r$ matrix differential equation
\begin{align*}
\dot{\bar \bfS}(t) = \bfU_0^{\top}\bfF(t,\bfU_0\bar \bfS(t)\bfV_0^{\top})\bfV_0\, , \quad \bar \bfS(t_0) = \bfS_0\, .
\end{align*}
\item \textbf{Augment}: Set up the augmented coefficient matrix $\widehat{\bfS}_1 \in\mathbb{R}^{2r\times 2r}$ as
\begin{align}\label{eq:Sparallel1st}
\widehat{\bfS}_1 = \begin{pmatrix}
\bar{\bfS}(t_1) & \bfL(t_1)^{\top}\widetilde \bfV_1\\
\widetilde \bfU_1^{\top} \bfK(t_1) & \bf0
\end{pmatrix}\, .
\end{align}
\item \textbf{Truncate}:
        Compute the singular value decomposition $\; \wh\bfS_1= \wh \bfP \wh \bfSigma \wh \bfQ^\top$ where $\wh\bfSigma=\diag(\sigma_j)$. Truncate to the tolerance~$\vartheta$ by choosing the new rank $r_1\le 2r$ as the %number of singular values greater than $\vartheta$. 
		minimal number $r_1$ such that 
		\begin{equation}
		\biggl(\ \sum_{j=r_1+1}^{2r} \sigma_j^2 \biggr)^{1/2} \le \vartheta.
		\end{equation}
		Compute the new factors for the approximation of $\bfA(t_1)$ as follows:
Let $\bfS_1$ be the $r_1\times r_1$ diagonal matrix with the $r_1$ largest singular values and let $\bfP_1\in \R^{2r\times r_1}$ and $\bfQ_1\in \R^{2r\times r_1}$ contain the first $r_1$ columns of $\wh \bfP$ and $\wh \bfQ$, respectively. Finally, set $\bfU_1 = \wh \bfU_1 \bfP_1\in \R^{m\times r_1}$ and
$\bfV_1 = \wh{\bfV}_1 \bfQ_1 \in \R^{n\times r_1}$.\\
\end{enumerate}

The solution at time $t_1$ is then given as $\bfY_1 = \bfU_1 \bfS_1\bfV_1^{\top}$, and the process is repeated until a final time $T$. This algorithm provides a rank $r_k$ for every time point $t_k$. Note that the augmented BUG integrator's augmented coefficient matrix is replaced by \eqref{eq:Sparallel1st}, which can be assembled from information gathered in three fully parallel steps. The key observation that allows such a construction is that the sequential term $\widetilde \bfU_1^\top \widehat \bfF(t) \widetilde \bfV_1^\top$ is of order $O(h^2)$ and can, therefore, be approximated by zero to ensure a robust error bound of order one. Removing the computation of this term allows the computation of all remaining blocks of \eqref{eq:Sparallel1st} in parallel. To provide such an error bound, the following standard properties \cite{KieriLubichWalach} are assumed:
\begin{assumption}\label{as:assumptions}
The right-hand side and the initial condition fulfill the following properties:
\begin{enumerate}
	\item $\bfF$ is Lipschitz-continuous, i.e., for all $\bfZ, \widetilde{\bfZ} \in \mathbb{R}^{m \times n}$ and $0\le t \le T$,
    $$
    \| \bfF(t, \bfZ) - \bfF(t, \widetilde{\bfZ}) \| 
    \leq
    L \| \bfZ - \widetilde{\bfZ} \|.
    $$	
    \item $\bfF$ is bounded, i.e., for all $\bfZ \in \mathbb{R}^{m \times n}$ and $0\le t \le T$, it holds $\| \bfF(t, \bfZ) \| \leq B$.
    \item Let $\mathcal{M}_r$ be the manifold of rank $r$ matrices, where $r\in\mathbb{N}$ is arbitrary. For $\bfY\in\mathcal{M}_r$, the normal part of $\bfF(t, \bfY)$, that is $\bfR(t, \bfY) := (\bfI - \bfP(\bfY))\bfF(t, \bfY)$ is $\varepsilon_r$-small, i.e., $\| \bfR(t, \bfY) \| \leq \varepsilon_r$. Hence, if $\wh\bfY\in\mathcal{M}_{2r}$, then $\| \bfR(t, \wh\bfY) \| \leq \varepsilon_{2r}$ where commonly one can assume $\varepsilon_{2r} \ll \varepsilon_{r}$.
    \item
    The error in the initial value is $\delta$-small:
    $$
    \| \bfY_0 - \bfA_0 \| \le \delta.
    $$\\
\end{enumerate}	
\end{assumption}

Then, the parallel integrator fulfills the following first-order robust error bound.
\begin{theorem}[\cite{CeKL23}, Theorem~4.1] Under the above Assumptions~\ref{as:assumptions}, the error of the parallel rank-adaptive integrator is bounded by
$$
\| \bfY_k - \bfA(t_k)\|  \le C_0 \delta + C_1 h + C_2 \eps_r + C_3  k\vartheta , \qquad 0\le kh \le T,
$$
where constants only depend on the Lipschitz constant and bound of $\bfF$, a bound of second derivatives of the exact solution, and an upper bound of the time stepsize.
\end{theorem}

Moreover, as pointed out in \cite[Section~3.3]{CeKL23}, the parallel integrator allows for a step rejection strategy that estimates the normal component of the vector field as $\Vert\widetilde \bfU_1^\top \bfF(\bfY_0) \widetilde \bfV_1^\top\Vert$. This strategy can also be used for the augmented BUG integrator and allows a further rank increase when the normal component is not sufficiently small.

\section{Second-order parallel integrators} \label{sec:2ndorderpar}
In the following, we extend the parallel integrator to order two for a matrix ODE
\begin{align*}
\dot \bfA(t) = \bfF(t,\bfA(t))\;.
\end{align*}
Constructing such an integrator seems non-trivial since the parallel integrator is designed as an approximation of order one. Specifically, setting the term $\widetilde \bfU_1^\top \wh\bfF(t) \widetilde \bfV_1$ in \eqref{eq:augmentedSMatrix} to zero yields a scheme of order one, whereas an accurate approximation of this term in turn leads to sequential update steps. 
The main idea of constructing a high-order parallel integrator is to determine basis matrices $\wh \bfU_1 = [\wh \bfU_0, \widetilde \bfU_2]$, and $\widehat \bfV_1 = [\widehat \bfV_0, \widetilde \bfV_2]$ such that if $\bfA(t_0)\in\mathcal{M}_r$, then
\begin{align*}
    \widehat \bfU_1^\top \bfA(t_1) \widehat \bfV_1 =
\begin{pmatrix} 
\wh\bfU_0^\top \bfA(t_1) \wh\bfV_0 & \wh\bfU_0^\top \bfA(t_1) \widetilde \bfV_2
\\
\widetilde \bfU_2^\top \bfA(t_1) \wh\bfV_0 & \widetilde \bfU_2^\top \bfA(t_1) \widetilde \bfV_2
\end{pmatrix}
\end{align*}
with $\Vert \widetilde \bfU_2^\top \bfA(t_1) \widetilde \bfV_2 \Vert \lesssim h^3 + h \varepsilon_r$ and $\Vert\bfA(t_1) - \widehat \bfU\widehat \bfU^\top \bfA(t_1) \widehat \bfV\widehat \bfV\Vert \lesssim h^3 + h \varepsilon_r$. It can be shown that such a basis exists and takes the form
\begin{alignat*}{2}
    \wh\bfU_1 =\,& \text{orth}([\wh\bfU_0, \bfK(t_1)\wh\bfV_0^{\top}\bfV_{\star}]) \quad &&\text{ with } \wh\bfU_0 = \text{orth}([\bfU_0, \bfF(t_0,\bfY_0)\bfV_0])\,, \\
    \wh\bfV_1 =\,& \text{orth}([\wh\bfV_0, \bfL(t_1)\wh\bfU_0^{\top}\bfU_{\star}]) \quad &&\text{ with } \wh\bfV_0 = \text{orth}([\bfV_0, \bfF(t_0,\bfY_0)^{\top}\bfU_0])\,,
\end{alignat*}
where $\bfK(t_1)$ and $\bfL(t_1)$ are the standard $K$ and $L$ steps, but with augmented basis matrices $\wh\bfU_0$ and $\wh\bfV_0$, and $\bfU_{\star}$ and $\bfV_{\star}$ are first-order approximations of the basis at the half-step. Such an integrator is proposed in Section~\ref{sec:v1}. Moreover, it turns out that instead using the basis
\begin{alignat*}{2}
    \wh\bfU_1 =\,& \text{orth}([\wh\bfU_0, \bfK(t_1)])\,, \quad 
    \wh\bfV_1 =\,& \text{orth}([\wh\bfV_0, \bfL(t_1)]) \,
\end{alignat*}
with the same $\wh\bfU_0$, $\wh\bfV_0$, $\bfK(t_1)$, and $\bfL(t_1)$ as before leads to an improved dependence on the normal component of the flow, similar to the $4r$ rank integrator of \cite{ceruti2024robust}. Specifically, instead of having $O(\varepsilon_r)$ terms in the global error bound, we get $O(\varepsilon_{2r} + h\varepsilon_r)$ terms instead, where commonly we can assume $\varepsilon_{2r}\ll \varepsilon_r$. This integrator shares a similar complexity as the first variant, however requires a singular value decomposition of a $4r \times 4r$ matrix.
In the following, we state these two variants of the second-order parallel integrator and discuss their computational footprint.

\subsection{Variant 1}\label{sec:v1}
The algorithm updates the solution $\bfY_0$ in factorized form $\bfU_0\in\mathbb{R}^{m\times r_0}$, $\bfS_0\in\mathbb{R}^{r_0\times r_0}$, $\bfV_0\in\mathbb{R}^{n\times r_0}$ to $\bfU_1\in\mathbb{R}^{m\times r_1}$, $\bfS_1\in\mathbb{R}^{r_1\times r_1}$, $\bfV_1\in\mathbb{R}^{n\times r_1}$. I.e., the solution is updated in factorized form from time $t_0$ to time $t_1 = t_0 + h$. Let us denote $r_0$ as $r$ in the following integrator:
\begin{enumerate}
\item Construct augmented basis matrices $\wh\bfU_0 = \text{orth}([\bfU_0, \bfF(t_0,\bfY_0)\bfV_0])$ and $\wh\bfV_0 = \text{orth}([\bfV_0, \bfF(t_0,\bfY_0)^{\top}\bfU_0])$.\label{step:init_augment}
\item Determine the augmented basis matrices $\widehat{\mathbf{U}}_1\in \mathbb{R}^{m\times 3r}$ and $\widehat{\mathbf{V}}_1\in \mathbb{R}^{n\times 3r}$ as well as the coefficient matrix $\bar{\mathbf{S}}(t_1) \in\mathbb{R}^{2r \times 2r}$ (in parallel):
\\[2mm]
\textbf{K-step}: Integrate from $t=t_0$ to $t_1$ the $m \times 2r$ matrix differential equation
\begin{align*}
\dot{\bfK}(t) =\,& \bfF(t,\bfK(t)\widehat \bfV_0^{\top})\widehat \bfV_0\, , \quad \bfK(t_0) = \bfU_0 \bfS_0 \bfV_0^{\top}\widehat{\bfV}_0\, .
\end{align*}
Given $\bfV_{\star} = \text{orth}(\bfY_0^{\top}\bfU_0 + \frac12 h \bfF(t_0,\bfY_0)^{\top}\bfU_0)$ construct $\widehat \bfU_1= [\widehat \bfU_0, \widetilde \bfU_2] = \text{orth}([\widehat \bfU_0, \bfK(t_1)\widehat{\bfV}_0^{\top}\bfV_{\star}])$. 
\\[2mm]
\textbf{L-step}: Integrate from $t=t_0$ to $t_1$ the $n \times 2r$ matrix differential equation
\begin{align*}
\dot{\bfL}(t) =\,& \bfF(t,\widehat \bfU_0 \bfL(t)^{\top})^{\top}\widehat \bfU_0\, , \quad \bfL(t_0) = \bfV_0 \bfS_0^{\top} \bfU_0^{\top}\widehat{\bfU}_0\, .
\end{align*}
Given $\bfU_{\star} = \text{orth}(\bfY_0\bfV_0 + \frac12 h \bfF(t_0,\bfY_0)\bfV_0)$ construct the time-updated basis $\widehat \bfV_1= [\widehat \bfV_0, \widetilde \bfV_2] = \text{orth}([\widehat \bfV_0, \bfL(t_1)\widehat \bfU_0^{\top}\bfU_{\star}])$. 
\\[2mm]
\textbf{S-step}: Integrate from $t=t_0$ to $t_1$ the $2r \times 2r$ matrix differential equation
\begin{align*}
\dot{\bar \bfS}(t) = \widehat \bfU_0^{\top}\bfF(t,\widehat{\bfU}_0\bar \bfS(t)\widehat{\bfV}_0^{\top})\widehat \bfV_0\, , \quad \bar \bfS(t_0) = \widehat{\bfU}_0^{\top}\bfU_0 \bfS_0 \bfV_0^{\top}\widehat{\bfV}_0\, .
\end{align*}
\item \textbf{Augment}: Set up the augmented coefficient matrix $\widehat{\bfS}_1 \in\mathbb{R}^{3r\times 3r}$ as
\begin{align*}
\widehat{\bfS}_1 = \begin{pmatrix}
\bar{\bfS}(t_1) & \bfL(t_1)^{\top}\widetilde \bfV_2\\
\widetilde \bfU_2^{\top} \bfK(t_1) & \bf0
\end{pmatrix}\, .
\end{align*}
\item \textbf{Truncate}:
        Compute the singular value decomposition $\; \wh\bfS_1= \wh \bfP \wh \bfSigma \wh \bfQ^\top$ where $\wh\bfSigma=\diag(\sigma_j)$. Truncate to the tolerance~$\vartheta$ by choosing the new rank $r_1\le 3r$ as the %number of singular values greater than $\vartheta$. 
		minimal number $r_1$ such that 
		\begin{equation}
		\biggl(\ \sum_{j=r_1+1}^{3r} \sigma_j^2 \biggr)^{1/2} \le \vartheta.
		\end{equation}
		Compute the new factors for the approximation of $\bfA(t_1)$ as follows:
Let $\bfS_1$ be the $r_1\times r_1$ diagonal matrix with the $r_1$ largest singular values and let $\bfP_1\in \R^{3r\times r_1}$ and $\bfQ_1\in \R^{3r\times r_1}$ contain the first $r_1$ columns of $\wh \bfP$ and $\wh \bfQ$, respectively. Finally, set $\bfU_1 = \wh \bfU_1 \bfP_1\in \R^{m\times r_1}$ and
$\bfV_1 = \wh{\bfV}_1 \bfQ_1 \in \R^{n\times r_1}$.\\
\end{enumerate}

The approximation after one time step is given by 
\begin{equation}\label{Y1-rabug}
	 \bfY_1 = \bfU_1 \bfS_1 \bfV_1^\top \approx \bfY(t_1).
\end{equation}
Following \cite[Section~3.3]{CeKL23}, this integrator allows for a step rejection strategy that estimates the normal component of the vector field as $\Vert\widetilde \bfU_2^\top \bfF(\bfY_0) \widetilde \bfV_2^\top\Vert$. 
Before presenting another variant of a second--order parallel integrator, we wish to understand the computation demands of the above algorithm.

\subsection{Computational costs}
A crucial factor in the overall computational costs is the time integration of the three matrix ODEs for $\bfK$, $\bfL$, and $\bfS$. To discuss the costs of constructing the right-hand sides, let us look at a simple matrix ODE with a right-hand side
\begin{align}\label{eq:F}
    \bfF(t, \bfY) = \sum_{\ell = 1}^M\bfC_{\ell}(t)\bfY\bfD_{\ell}(t)\,
\end{align}
where $\bfC_{\ell}$ and $\bfD_{\ell}$ are sparse matrices with $c_{\ell}\cdot m$ entries in $\bfC_{\ell}$ and $d_{\ell}\cdot n$ entries in $\bfD_{\ell}$. Such problems are widespread and arise, for example, in computational radiation therapy \cite{kusch2023robust}, where $\bfC_{\ell}$ are sparse stencil matrices (which are commonly independent of time) and $\bfD_{\ell}$ are sparse flux matrices (which can be time-dependent). Then, with $\bfY = \bfK\bfV^{\top}\in\mathbb{R}^{m\times n}$, where $\bfK\in\mathbb{R}^{m\times \widehat r}$, $\bfV\in\mathbb{R}^{n\times \widehat r}$ with a generic rank $\widehat r$, the projection $\bfF(t,\bfY)\bfV$ reads
\begin{align*}
    \bfF(t,\bfY)\bfV = \sum_{\ell = 1}^M\bfC_{\ell}(t)\bfK (\bfV^{\top}\bfD_{\ell}(t)\bfV)\,.
\end{align*}
The computation of $\widetilde{\bfD}_{\ell}(t) := \bfV^{\top}\bfD_{\ell}(t)\bfV$ requires $\widehat r^2 \cdot n \cdot d_{\ell}$ operations. The computation of $\bfK\widetilde{\bfD}_{\ell}(t)$ then requires $\widehat r^2 \cdot m$ and the multiplication with $\bfC_{\ell}(t)$ requires $\widehat r \cdot m \cdot c_{\ell}$ operations. Hence, a single computation of $\bfF(t_0,\bfY_0)\bfV_0$ as required in step \ref{step:init_augment} with $\widehat r \equiv r$ has computational costs 
\begin{align*}
    C_{\mathrm{aug}}^K := \sum_{\ell = 1}^M r \cdot (r\cdot n \cdot d_{\ell} + r\cdot m + m \cdot c_{\ell}) \,.
\end{align*}
In the same manner, the computation of $\bfF(t_0,\bfY_0)^{\top}\bfU_0$ requires
\begin{align*}
    C_{\mathrm{aug}}^L := \sum_{\ell = 1}^M r \cdot (r\cdot m \cdot c_{\ell} + r\cdot n + n \cdot d_{\ell}) \,
\end{align*}
operations. Solving the $K$ and $L$ steps can become substantially more expensive if we use a time-integration method that performs $N_{\mathrm{ode}}$ evaluations of the right-hand side. Then, with $\wh r \equiv 2r$, we have
\begin{align*}
    C_{\mathrm{ode}}^K :=\,& N_{\mathrm{ode}}\cdot \sum_{\ell = 1}^M r \cdot (4r\cdot n \cdot d_{\ell} + 4r\cdot m + 2m \cdot c_{\ell}) \,,\\
    C_{\mathrm{ode}}^L :=\,& N_{\mathrm{ode}}\cdot \sum_{\ell = 1}^M r \cdot (4r\cdot m \cdot c_{\ell} + 4r\cdot n + 2n \cdot d_{\ell}) \,.
\end{align*}
Moreover, with $\bfY = \wh\bfU_0\bfS\wh\bfV_0^{\top}$, the $S$-step requires the computation of
\begin{align*}
    \wh\bfU_0^{\top}\bfF(t,\bfY)\wh\bfV_0 = \sum_{\ell = 1}^M(\wh\bfU_0^{\top}\bfC_{\ell}(t)\wh\bfU_0) \bfS (\wh\bfV_0^{\top}\bfD_{\ell}(t)\wh\bfV_0)\,.
\end{align*}
Here, the costs to compute $\widetilde{\bfC}_{\ell}(t) := \wh\bfU_0^{\top}\bfC_{\ell}(t)\wh\bfU_0$ and $\widetilde{\bfD}_{\ell}(t) := \wh\bfV_0^{\top}\bfD_{\ell}(t)\wh\bfV_0$ are $4 r^2 (n \cdot d_{\ell} + m \cdot c_{\ell})$ operations. Moreover, the computation of $\widetilde{\bfC}_{\ell}(t)\bfS(t)\widetilde{\bfD}_{\ell}(t)$ requires $16r^3$ operations. Hence, the costs to solve the $S$-step are
\begin{align*}
    C_{\mathrm{ode}}^S :=\,& N_{\mathrm{ode}}\cdot \sum_{\ell = 1}^M r \cdot (4r\cdot n \cdot d_{\ell} + 4r\cdot m \cdot c_{\ell} + 16 r^2) \,.
\end{align*}
Lastly, the computation of the singular value decomposition to truncate the solution requires $O(r^3)$, and the orthonormalizations require $O(nr^2)$ and $O(mr^2)$ operations, respectively. Note that for this choice of the right-hand side, the computation of the $S$-step is usually the most expensive part. Moreover, when the flux matrices in \eqref{eq:F} are time-independent, the construction of $\widetilde{\bfC}_{\ell}$ and $\widetilde{\bfD}_{\ell}$ is required only once per time step of the parallel integrator.

\subsection{Variant 2}\label{sec:v2}
An alternative approach that does not require the computation of $\bfU_{\star}$ and $\bfV_{\star}$ and yields an improved error bound but uses a rank $4r$ augmented coefficient matrix (and requires its singular value decomposition) is the following:
\begin{enumerate}
\item Construct augmented basis matrices $\wh\bfU_0 = \text{orth}([\bfU_0, \bfF(t_0,\bfY_0)\bfV_0])$ and $\wh\bfV_0 = \text{orth}([\bfV_0, \bfF(t_0,\bfY_0)^{\top}\bfU_0])$.
\item Determine the augmented basis matrices $\widehat{\mathbf{U}}_1\in \mathbb{R}^{m\times 4r}$ and $\widehat{\mathbf{V}}_1\in \mathbb{R}^{n\times 4r}$ as well as the coefficient matrix $\bar{\mathbf{S}}(t_1) \in\mathbb{R}^{2r \times 2r}$ (in parallel):
\\[2mm]
\textbf{K-step}: Integrate from $t=t_0$ to $t_1$ the $m \times 2r$ matrix differential equation
\begin{align*}
\dot{\bfK}(t) =\,& \bfF(t,\bfK(t)\widehat \bfV_0^{\top})\widehat \bfV_0\, , \quad \bfK(t_0) = \bfU_0 \bfS_0 \bfV_0^{\top}\widehat{\bfV}_0\, .
\end{align*}
Construct $\widehat \bfU_1= [\widehat \bfU_0, \widetilde \bfU_2] = \text{orth}([\widehat \bfU_0, \bfK(t_1)])$. 
\\[2mm]
\textbf{L-step}: Integrate from $t=t_0$ to $t_1$ the $n \times 2r$ matrix differential equation
\begin{align*}
\dot{\bfL}(t) =\,& \bfF(t,\widehat \bfU_0 \bfL(t)^{\top})^{\top}\widehat \bfU_0\, , \quad \bfL(t_0) = \bfV_0 \bfS_0^{\top} \bfU_0^{\top}\widehat{\bfU}_0\, .
\end{align*}
Construct $\widehat \bfV_1= [\widehat \bfV_0, \widetilde \bfV_2] = \text{orth}([\widehat \bfV_0, \bfL(t_1)])$. 
\\[2mm]
\textbf{S-step}: Solve the $2r \times 2r$ matrix differential equation
\begin{align*}
\dot{\bar \bfS}(t) = \widehat \bfU_0^{\top}\bfF(t,\widehat{\bfU}_0\bar \bfS(t)\widehat{\bfV}_0^{\top})\widehat \bfV_0\, , \quad \bar \bfS(t_0) = \widehat{\bfU}_0^{\top}\bfU_0 \bfS_0 \bfV_0^{\top}\widehat{\bfV}_0\, .
\end{align*}
\item \textbf{Augment}: Set up the augmented coefficient matrix $\widehat{\bfS}_1 \in\mathbb{R}^{4r\times 4r}$ as
\begin{align*}
\widehat{\bfS}_1 = \begin{pmatrix}
\bar{\bfS}(t_1) & \bfL(t_1)^{\top}\widetilde \bfV_2\\
\widetilde \bfU_2^{\top} \bfK(t_1) & \bf0
\end{pmatrix}\, .
\end{align*}
\item \textbf{Truncate}:
        Compute the singular value decomposition $\; \wh\bfS_1= \wh \bfP \wh \bfSigma \wh \bfQ^\top$ where $\wh\bfSigma=\diag(\sigma_j)$. Truncate to the tolerance~$\vartheta$ by choosing the new rank $r_1\le 4r$ as the %number of singular values greater than $\vartheta$. 
		minimal number $r_1$ such that 
		\begin{equation}
		\biggl(\ \sum_{j=r_1+1}^{4r} \sigma_j^2 \biggr)^{1/2} \le \vartheta.
		\end{equation}
		Compute the new factors for the approximation of $\bfA(t_1)$ as follows:
Let $\bfS_1$ be the $r_1\times r_1$ diagonal matrix with the $r_1$ largest singular values and let $\bfP_1\in \R^{4r\times r_1}$ and $\bfQ_1\in \R^{4r\times r_1}$ contain the first $r_1$ columns of $\wh \bfP$ and $\wh \bfQ$, respectively. Finally, set $\bfU_1 = \wh \bfU_1 \bfP_1\in \R^{m\times r_1}$ and
$\bfV_1 = \wh{\bfV}_1 \bfQ_1 \in \R^{n\times r_1}$.\\
\end{enumerate}

The main difference to the integrator in Section~\ref{sec:v1} is the use of the entire span of $\bfK(t_1)$ and $\bfL(t_1)$ to construct $\widetilde\bfU_2$ and $\widetilde\bfV_2$, respectively. Note, however, that the augmentation step and solving the ordinary differential equations for $\bfK(t)$, $\bfL(t)$, and $\bar{\bfS}(t)$ remain the same, leading to comparable computational costs. %The following derivations of the robust error bound for these integrators use the fact that both variants span $\bfK(t_1)\widehat \bfV_0^{\top}\bfV_{\star}$ and $\bfL(t_1)\widehat \bfU_0^{\top}\bfU_{\star}$ exactly.

\section{Robust error bound}\label{sec:robusterror}
The proposed parallel integrators have a second-order robust error bound. To derive this bound, we use the same notation for both parallel integrators in Sections~\ref{sec:v1} and \ref{sec:v2}; however, we state in each lemma to which integrator we refer. In the remainder, constants are denoted by either upper or lower-case $C$ variables, and we reuse variable names of constants for ease of presentation. All constants only depend on the Lipschitz constant and bound of $\bfF$, a bound of third derivatives of the exact solution, and an upper bound of the time stepsize. That is, there is no dependence on the curvature of $\mathcal{M}$ or $\rho^{-1}$. Such dependencies would arise when for example Taylor expanding the bases or the projected dynamics, which is avoided in the following derivations. Let us first state the main result:
\begin{theorem}\label{th:robustglobal} Under the above Assumptions~\ref{as:assumptions}, the error of the parallel integrator of Section~\ref{sec:v1} is bounded by
$$
\| \bfY_k - \bfA(t_k)\|  \le C_0^{(1)} \delta + C_1^{(1)} h^2 + C_2^{(1)} \eps_r + C_3^{(1)}  k\vartheta , \qquad 0\le kh \le T,
$$
whereas the error of the parallel integrator of Section~\ref{sec:v2} is bounded by
$$
\| \bfY_k - \bfA(t_k)\|  \le C_0^{(2)} \delta + C_1^{(2)} h^2 + C_2^{(2)} \eps_{2r} + C_3^{(2)} h\eps_r + C_4^{(2)}  k\vartheta , \qquad 0\le kh \le T\,.
$$
Here, constants $C_i^{(j)}$ only depend on the Lipschitz constant and bound of $\bfF$, a bound of third derivatives of the exact solution, and an upper bound of the time stepsize.
\end{theorem}
To prove this theorem, let us start by bounding the local error $\Vert \widetilde\bfA(t_1) - \bfY_1 \Vert$, where $\bfY_1$ is the solution obtained with the parallel integrator described in either Section~\ref{sec:v1} or Section~\ref{sec:v2}, respectively. $\widetilde \bfA(t)$ is the solution to the full problem with a low-rank initial condition, i.e.,
\begin{align*}
\dot{\widetilde \bfA}(t) = \bfF(t,\widetilde \bfA(t)), \qquad \widetilde{\bfA}(t_0) = \bfY_0 = \bfU_0\bfS_0\bfV_0^{\top}\;.
\end{align*}
Moreover, the solution before truncation is denoted as $\wh \bfY_1 = \wh\bfU_1 \wh\bfS_1\wh\bfV_1^{\top}$. To bound the local error, we use a triangular inequality
\begin{align}\label{eq:boundGlobal}
\Vert \widetilde \bfA(t_1) - \bfY_1 \Vert \leq\,& \Vert \widetilde \bfA(t_1) - \widehat \bfU_1\widehat{\bfU}_1^{\top}\widetilde \bfA(t_1)\widehat \bfV_1\widehat{\bfV}_1^{\top} \Vert + \Vert \widehat \bfU_1\widehat{\bfU}_1^{\top}\widetilde \bfA(t_1)\widehat \bfV_1\widehat{\bfV}_1^{\top} - \wh\bfY_1 \Vert \nonumber \\
\,& + \Vert \wh\bfY_1 - \bfY_1 \Vert\;.
\end{align}
Let us bound all individual terms on the right-hand side of \eqref{eq:boundGlobal}. We start with the projected full-rank solution $\widetilde \bfA(t_1)$ using the basis generated with the parallel integrator described in Section~\ref{sec:v2}:
%The distance between the full rank and the projected solution fulfills
%\begin{lemma}\label{le:AX}
%    Given the Assumptions~\ref{as:assumptions}, the first term in \eqref{eq:boundGlobal} is bounded by
%    \begin{align*}
%        \Vert \bfA(t) - \bfX(t) \Vert \leq \varepsilon_r \int_{0}^h e^{L_1\xi}\,d\xi\,.
%    \end{align*}
%\end{lemma}
%The proof of this lemma is standard and can, for example, be found in \cite[Lemma~1]{KiV19}. Next, we look at the second term in \eqref{eq:boundGlobal}:
\begin{lemma}\label{le:v2Proj}
Given the Assumptions~\ref{as:assumptions}, the full solution projected onto the basis of the parallel integrator from Section~\ref{sec:v2} fulfills
\begin{align*}
\Vert \widetilde \bfA(t_1) - \widehat \bfU_1\widehat{\bfU}_1^{\top}\widetilde \bfA(t_1)\widehat \bfV_1\widehat{\bfV}_1^{\top} \Vert \leq C_1 h^3 + C_2\varepsilon_{2r} h+ C_3\varepsilon_{r} h^2\;,
\end{align*}
where constants $C_i$ only depend on the Lipschitz constant and bound of $\bfF$, a bound of third derivatives of the exact solution, and an upper bound of the time stepsize.
\end{lemma}
\begin{proof}
First, note that
\begin{align*}
\Vert \widetilde \bfA(t_1) - \widehat \bfU_1\widehat{\bfU}_1^{\top}\widetilde \bfA(t_1)\widehat \bfV_1\widehat{\bfV}_1^{\top} \Vert \leq\,& \Vert \widetilde \bfA(t_1) - \widehat \bfU_1\widehat{\bfU}_1^{\top}\widetilde \bfA(t_1) \Vert\\
\,&+ \Vert \widehat \bfU_1\widehat{\bfU}_1^{\top}\widetilde \bfA(t_1) - \widehat \bfU_1\widehat{\bfU}_1^{\top}\widetilde \bfA(t_1)\widehat \bfV_1\widehat{\bfV}_1^{\top} \Vert\\
\leq\,& \Vert (\bfI - \widehat \bfU_1\widehat{\bfU}_1^{\top})\widetilde \bfA(t_1) \Vert + \Vert \widetilde \bfA(t_1)(\bfI - \widehat \bfV_1\widehat{\bfV}_1^{\top}) \Vert\, .
\end{align*}
We start by bounding the first term on the right-hand side. For this, let us denote the rank $r$ approximation obtained with the fixed-rank BUG integrator at time $t_{\sfrac12}:= t_0 + \frac12 h$ with a forward Euler time discretization in each substep by $\wh\bfY_{\star}$ with
\begin{subequations}
    \begin{align}
        \wh\bfY_{\star} =\,& \wh\bfU_{0}\wh\bfU_{0}^{\top}\left(\bfY_0 + \frac12h \bfF(t_0,\bfY_0)\right)\wh\bfV_{0}\wh\bfV_{0}^{\top} \,.\label{eq:Ystar}
    \end{align}
\end{subequations}
Note that $\wh\bfU_0$ indeed corresponds to the basis obtained with the augmented BUG integrator using forward Euler time steps since 
\begin{align*}
    \text{range}(\wh\bfU_0) = \text{range}([\bfU_0, \bfU_0\bfS_0 + \frac{h}{2} \bfF(t_0, \bfY_0)\bfV_0])= \text{range}([\bfU_0, \bfF(t_0, \bfY_0)\bfV_0])\,.
\end{align*}
The same holds for the basis of the co-range, $\wh\bfV_0$. From the first-order robust error bound of the augmented BUG integrator, we know that $\Vert \wh\bfY_{\star} - \widetilde{\bfA}(t_{\sfrac12})\Vert \lesssim h^2 + h\varepsilon_r$. According to the midpoint rule, we then have, similar to \cite[Theorem~1]{ceruti2024robust},
\begin{align*}
\Vert (\bfI - \widehat \bfU_1\widehat{\bfU}_1^{\top})\widetilde \bfA(t_1) \Vert \leq\,& \left\Vert (\bfI-\widehat \bfU_1\widehat{\bfU}_1^{\top})\left( \widetilde \bfA(t_0) + h\bfF(t_{\sfrac12},\widetilde\bfA(t_{\sfrac12})) \right) \right\Vert + c_1 h^3\\
\leq\,& \left\Vert (\bfI-\widehat \bfU_1\widehat{\bfU}_1^{\top})\left( \bfY_0 + h\bfF(t_{\sfrac12},\wh\bfY_{\star}) \right) \right\Vert + c_2 h^3 + c_3 h^2\varepsilon_r \,.
\end{align*}
Note that the first term on the right side is $(\bfI-\widehat \bfU_1\widehat{\bfU}_1^{\top})\bfY_0 = 0$. To bound the second term, denote the projector onto the nullspace of $\widehat \bfU_1$ as $\bfP_{\widehat \bfU_1}^{\perp}:=(\bfI-\widehat \bfU_1\widehat{\bfU}_1^{\top})$. Then, we have with $\bfP_{\widehat \bfU_1}^{\perp}\wh\bfU_{0} = 0$ that
\begin{align*}
h\left\Vert \bfP_{\widehat \bfU_1}^{\perp}\bfF(t_{\sfrac12},\wh\bfY_{\star}) \right\Vert \leq h  \left\Vert \bfP_{\widehat \bfU_1}^{\perp}\bfM(t_{\sfrac12},\wh\bfY_{\star}) \right\Vert + h \varepsilon_{2r}= h\left\Vert \bfP_{\widehat \bfU_1}^{\perp}\bfF(t_{\sfrac12},\wh\bfY_{\star})\wh\bfV_{0}\wh\bfV_{0}^{\top} \right\Vert +h\varepsilon_{2r}\,.
\end{align*}
By the midpoint rule, defining $\bfY_{\sfrac12}:= \bfY_0 +\frac12h\bfF(t_0,\bfY_0)\widehat{\bfV}_0\widehat{\bfV}_0^{\top}$ gives
\begin{align*}
    \bfK(t_1) = \bfK(t_0) + h \bfF(t_{\sfrac12},\bfY_{\sfrac12})\widehat{\bfV}_0^{\top} + O(h^3)\,.
\end{align*}
Therefore, since by the construction of the augmented basis in the parallel integrator presented in Section~\ref{sec:v2} the term $\bfK(t_1)$ is spanned by $\widehat \bfU_1$,
\begin{align*}
h\left\Vert \bfP_{\widehat \bfU_1}^{\perp}\bfF(t_{\sfrac12},\wh\bfY_{\star})\wh\bfV_{0}\wh\bfV_{0}^{\top} \right\Vert = \,& h\left\Vert \bfP_{\widehat \bfU_1}^{\perp}\left(\bfF(t_{\sfrac12},\wh\bfY_{\star}) - \frac{1}{h}\bfK(t_1)\widehat \bfV_0^{\top}\right)\wh\bfV_{0}\right\Vert\, \\
\leq \,& h\left\Vert \bfP_{\widehat \bfU_1}^{\perp}\left(\bfF(t_{\sfrac12},\wh\bfY_{\star}) - \frac{1}{h}\bfK(t_0)\wh{\bfV}_0^{\top} -\bfF\left(t_{\sfrac12},\bfY_{\sfrac12}\right)\widehat \bfV_0\widehat{\bfV}_0^{\top}\right)\wh\bfV_{0}\right\Vert + c_4h^3\,.
\end{align*}
With $\bfP_{\widehat \bfU_1}^{\perp}\bfK(t_0) = 0$, we have
\begin{align*}
h\left\Vert \bfP_{\widehat \bfU_1}^{\perp}\bfF(t_{\sfrac12},\wh\bfY_{\star})\wh\bfV_{0}\wh\bfV_{0}^{\top} \right\Vert
 \leq \,& h\left\Vert \bfF(t_{\sfrac12},\wh\bfY_{\star}) - \bfF\left(t_{\sfrac12},\bfY_{\sfrac12}\right)\right\Vert + c_4h^3\, \\
\leq \,& hL\left\Vert \wh\bfY_{\star} - \bfY_{\sfrac12}\right\Vert +c_4h^3\, \\
= \,& \frac{h^2L}{2}\left\Vert (\bfI - \wh\bfU_0\wh\bfU_0^{\top})\bfF(t_0, \bfY_0)\wh \bfV_0\wh\bfV_0^{\top} \right\Vert +c_4h^3\, \\
\leq \,& \frac12 h^2 L \varepsilon_r + c_4h^3\, .
\end{align*}
Hence, $\Vert (\bfI - \widehat \bfU_1\widehat{\bfU}_1^{\top})\widetilde \bfA(t_1) \Vert \lesssim h^2\varepsilon_r + h\varepsilon_{2r} + h^3$. The analogous derivation for $\Vert \widetilde \bfA(t_1)(\bfI - \widehat \bfV_1\widehat{\bfV}_1^{\top}) \Vert$ yields the desired bound.
\end{proof}
Next, the projected full-rank solution $\widetilde \bfA(t_1)$ using the basis generated with the parallel integrator described in Section~\ref{sec:v1} has a similar bound, however, with a stronger contribution from the normal component:
\begin{lemma}
Given the Assumptions~\ref{as:assumptions}, the full solution projected onto the basis computed with the parallel integrator of Section~\ref{sec:v1} fulfills
\begin{align*}
\Vert \widetilde \bfA(t_1) - \widehat \bfU_1\widehat{\bfU}_1^{\top}\widetilde \bfA(t_1)\widehat \bfV_1\widehat{\bfV}_1^{\top} \Vert \leq C_1 h^3 + C_2\varepsilon_{r} h\;
\end{align*}
where constants $C_i$ only depend on the Lipschitz constant and bound of $\bfF$, a bound of third derivatives of the exact solution, and an upper bound of the time stepsize.
\end{lemma}
\begin{proof}
The proof is similar to that of Lemma~\ref{le:v2Proj}. Here, instead of $\wh\bfY_{\star}$, we use the fixed-rank BUG approximation at the half-step
    \begin{subequations}
    \begin{align}
        \bfU_{\star} =\,& \text{orth}(\bfK_0 + \frac12 h \bfF(t_0,\bfY_0)\bfV_0)\\
        \bfV_{\star} =\,& \text{orth}(\bfL_0 + \frac12 h \bfF(t_0,\bfY_0)^{\top}\bfU_0)\\
        \bfY_{\star} =\,& \bfU_{\star}\bfU_{\star}^{\top}\left(\bfY_0 + \frac12h \bfF(t_0,\bfU_{\star}\bfU_{\star}^{\top}\bfY_0\bfV_{\star}\bfV_{\star}^{\top})\right)\bfV_{\star}\bfV_{\star}^{\top} \,.
    \end{align}
\end{subequations}
Then, the same derivation as before with $\wh\bfY_{\star}$ replaced by $\bfY_{\star}$ yields
\begin{align*}
h\left\Vert \bfP_{\widehat \bfU_1}^{\perp}\bfF(t_{\sfrac12},\bfY_{\star}) \right\Vert \leq h  \left\Vert \bfP_{\widehat \bfU_1}^{\perp}\bfM(t_{\sfrac12},\bfY_{\star}) \right\Vert + h \varepsilon_r= h\left\Vert \bfP_{\widehat \bfU_1}^{\perp}\bfF(t_{\sfrac12},\bfY_{\star})\bfV_{\star}\bfV_{\star}^{\top} \right\Vert +h\varepsilon_r\,.
\end{align*}
Therefore, since by the construction of the augmented basis in \emph{both} variants of the second--order parallel integrator, the term $\bfK(t_1)\widehat \bfV_0^{\top}\bfV_{\star}$ is spanned by $\widehat \bfU_1$,
\begin{align*}
h\left\Vert \bfP_{\widehat \bfU_1}^{\perp}\bfF(t_{\sfrac12},\bfY_{\star})\bfV_{\star}\bfV_{\star}^{\top} \right\Vert = \,& h\left\Vert \bfP_{\widehat \bfU_1}^{\perp}\left(\bfF(t_{\sfrac12},\bfY_{\star}) - \frac{1}{h}\bfK(t_1)\widehat \bfV_0^{\top}\right)\bfV_{\star}\right\Vert\, \\
\leq \,& h\left\Vert \bfP_{\widehat \bfU_1}^{\perp}\left(\bfF(t_{\sfrac12},\bfY_{\star}) - \frac{1}{h}\bfK(t_0)\widehat{\bfV}_0^{\top} -\bfF\left(t_{\sfrac12},\bfY_{\sfrac12}\right)\widehat \bfV_0\widehat{\bfV}_0^{\top}\right)\bfV_{\star}\right\Vert + c_4h^3\,.
\end{align*}
Since $\bfP_{\widehat \bfU_1}^{\perp}\bfK(t_0) = 0$ and $\wh\bfV_0\wh\bfV_0^{\top}\bfV_{\star} = \bfV_{\star}$ we have
\begin{align*}
h\left\Vert \bfP_{\widehat \bfU_1}^{\perp}\bfF(t_{\sfrac12},\bfY_{\star})\bfV_{\star}\bfV_{\star}^{\top} \right\Vert
 \leq \,& h\left\Vert \bfF(t_{\sfrac12},\bfY_{\star}) - \bfF\left(t_{\sfrac12},\bfY_{\sfrac12}\right)\right\Vert + c_4h^3\, \\
\leq \,& hL\left\Vert \bfY_{\star} - \bfY_{\sfrac12}\right\Vert +c_4h^3\, \\
\leq \,& hL\left\Vert \bfY_{\star} - \widetilde\bfA(t_{\sfrac12})\right\Vert + hL\left\Vert \widetilde\bfA(t_{\sfrac12}) - \bfY_{\sfrac12}\right\Vert +c_4h^3\, \\
\leq \,& 2h^2 L \varepsilon_r + c_5h^3\, .
\end{align*}
\end{proof}
Next, the second term in the local error \eqref{eq:boundGlobal} has the same robust error bound for both integrators of Sections~\ref{sec:v1} and \ref{sec:v2}:
\begin{lemma}
Given the Assumptions~\ref{as:assumptions}, the distance between the projected full dynamics and the solution of both parallel integrators of Sections~\ref{sec:v1} and \ref{sec:v2} is bounded by
\begin{align*}
\Vert \widehat \bfU_1\widehat{\bfU}_1^{\top}\widetilde \bfA(t_1)\widehat \bfV_1\widehat{\bfV}_1^{\top} -\wh\bfY_1 \Vert \leq \wh C_1 h\varepsilon_{2r} + \wh C_2 h^2\varepsilon_r + \wh C_3 h^3\, ,
\end{align*}
where $\wh C_i$ only depend on the Lipschitz constant and bound of $\bfF$, a bound of third derivatives of the exact solution, and an upper bound of the time stepsize.
\end{lemma}
\begin{proof}
By the definition of the augmentation step, we have
\begin{align*}
    \Vert \widehat \bfU_1\widehat{\bfU}_1^{\top}\widetilde \bfA(t_1)\widehat \bfV_1\widehat{\bfV}_1^{\top} -\wh\bfY_1 \Vert \leq\,& \Vert \widehat{\bfU}_0^{\top}\widetilde \bfA(t_1)\widehat \bfV_0 - \bar \bfS(t_1) \Vert + \Vert \widetilde{\bfU}_2^{\top}\widetilde \bfA(t_1)\widehat \bfV_0 - \widetilde{\bfU}_2^{\top}\bfK(t_1) \Vert \\
    \,&+ \Vert \widehat \bfU_0^{\top}\widetilde \bfA(t_1)\widetilde{\bfV}_2 - \bfL(t_1)^{\top}\widetilde{\bfV}_2 \Vert + \Vert \widetilde{\bfU}_2^{\top}\widetilde \bfA(t_1)\widetilde{\bfV}_2\Vert\,.
\end{align*}
First, the expansion coefficients with respect to $\widehat{\bfU}_0$ and $\widehat \bfV_0$ are bounded according to
\begin{align*}
\Vert \widehat{\bfU}_0^{\top}\widetilde \bfA(t_1)\widehat \bfV_0 - \bar \bfS(t_1) \Vert \leq\, & \int_{t_0}^{t_1} \Vert \widehat{\bfU}_0^{\top}\bfF( t,\widetilde \bfA(t) )\widehat \bfV_0 -\widehat{\bfU}_0^{\top}\bfF( t,\widehat \bfU_0\bar{\bfS}(t)\widehat \bfV_0^{\top} )\widehat \bfV_0  \Vert \, dt \\
\leq \,& L\int \Vert \widetilde \bfA(t)  - \widehat \bfU_0\bar{\bfS}(t)\widehat \bfV_0^{\top}  \Vert \, dt \\
\leq \,& L\int\int \Vert\bfF(\xi,\widetilde \bfA(\xi))  - \widehat \bfU_0\widehat \bfU_0^{\top}\bfF(\xi,\widehat \bfU_0\bar{\bfS}(\xi)\widehat \bfV_0^{\top})\widehat \bfV_0\widehat \bfV_0^{\top}  \Vert \, dtd\xi\\
\leq \,& Lh^2 \Vert\bfF(t_0,\bfY_0)  - \widehat \bfU_0\widehat \bfU_0^{\top}\bfF(t_0,\bfY_0)\widehat \bfV_0\widehat \bfV_0^{\top}  \Vert + \wh c_1 h^3\,.
\end{align*}
The norm on the right-hand side can be bounded by
\begin{align*}
\,& \Vert\bfF(t_0,\bfY_0)  - \widehat \bfU_0\widehat \bfU_0^{\top}\bfF(t_0,\bfY_0)\widehat \bfV_0\widehat \bfV_0^{\top}  \Vert\\
=\,& \Vert\bfF(t_0,\bfY_0)- \widehat \bfU_0\widehat \bfU_0^{\top}\bfF(t_0,\bfY_0) + \widehat \bfU_0\widehat \bfU_0^{\top}\left(\bfF(t_0,\bfY_0)  -\bfF(t_0,\bfY_0)\widehat \bfV_0\widehat \bfV_0^{\top}\right)  \Vert \\
\leq \,& \Vert (\bfI- \widehat \bfU_0\widehat \bfU_0^{\top})\bfF(t_0,\bfY_0)\Vert + \Vert\bfF(t_0,\bfY_0) (\bfI-\widehat \bfV_0\widehat \bfV_0^{\top})  \Vert \\
\leq \,& \Vert (\bfI- \widehat \bfU_0\widehat \bfU_0^{\top})\bfM(t_0,\bfY_0)\Vert + \Vert \bfM(t_0,\bfY_0) (\bfI-\widehat \bfV_0\widehat \bfV_0^{\top})  \Vert + 2\varepsilon_r = 2\varepsilon_r\,.
\end{align*}
Second, we bound expansion coefficients with respect to $\widetilde{\bfU}_2$ and $\widehat \bfV_0$
\begin{align*}
\Vert \widetilde{\bfU}_2^{\top}\widetilde \bfA(t_1)\widehat \bfV_0 - \widetilde{\bfU}_2^{\top}\bfK(t_1) \Vert \leq\, & \int \Vert \widetilde{\bfU}_2^{\top}\bfF(t,\widetilde \bfA(t))\widehat \bfV_0 -\widetilde{\bfU}_2^{\top}\bfF(t, \bfK(t)\widehat \bfV_0^{\top} )\widehat \bfV_0  \Vert \, dt \\
\leq \,& L\int \Vert \widetilde \bfA(t)  - \bfK(t)\widehat \bfV_0^{\top}  \Vert \, dt \\
\leq \,& L\int\int \Vert\bfF(\xi,\widetilde \bfA(\xi))  -\bfF(\xi,\bfK(\xi)\widehat \bfV_0^{\top})\widehat \bfV_0\widehat \bfV_0^{\top}  \Vert \, dtd\xi\\
\leq \,& L\int\int \Vert\bfF(t_0,\bfY_0)  -\bfF(t_0,\bfY_0)\widehat \bfV_0\widehat \bfV_0^{\top}  \Vert \, dtd\xi + \wh c_2h^3\\
\leq \,& Lh^2\varepsilon_r + \wh c_2h^3\,.
\end{align*}
Expansion coefficients with respect to $\widetilde{\bfV}_2$ and $\widehat \bfU_0$ can be bounded analogously, that is,
\begin{align*}
\Vert \widehat \bfU_0^{\top}\widetilde \bfA(t_1)\widetilde{\bfV}_2 - \bfL(t_1)^{\top}\widetilde{\bfV}_2 \Vert \leq Lh^2\varepsilon_r + \wh c_3 h^3 \, .
\end{align*}
Lastly, coefficients with respect to $\widetilde{\bfV}_2$ and $\widetilde{\bfU}_2$ are of order $\mathcal{O}(h\varepsilon_{2r} + h^2\varepsilon_{r} + h^3)$: The midpoint rule with $\wh\bfY_{\star}$ as defined in \eqref{eq:Ystar} gives
%Given $\bfY = \bfY_0+h\widehat \bfU_0\widehat \bfU_0^{\top}\bfF(t_0,\bfY_0)\widehat \bfV_0\widehat \bfV_0^{\top}$ we have with the midpoint rule 
%\begin{align*}
%\Vert \widetilde{\bfU}_2^{\top}\widetilde \bfA(t_1)\widetilde{\bfV}_2\Vert \leq\, & \left\Vert \widetilde{\bfU}_2^{\top}\left[\bfY_0 + \frac{h}{2}(\bfF(t_0,\bfY_0) + \bfF(t_1,\widetilde\bfY)) \right]\widetilde{\bfV}_2\right\Vert + \wh c_4h^3\\
%=\, & \frac{h}2\left\Vert \widetilde{\bfU}_2^{\top}\left[\bfF(t_0,\bfY_0) +\bfF(t_1, \widetilde\bfY  ) \right]\widetilde{\bfV}_2\right\Vert+ \wh c_4h^3\\
%\leq\, & \frac{h}2\left\Vert \widetilde{\bfU}_2^{\top}\left[ \bfM(t_0,\bfY_0) + \bfM(t_1,\widetilde\bfY) \right]\widetilde{\bfV}_2\right\Vert + h\varepsilon+ \wh c_4h^3\\
%=\,& h\varepsilon+ \wh c_4h^3\,.
%\end{align*}
\begin{align*}
\Vert \widetilde{\bfU}_2^{\top}\widetilde \bfA(t_1)\widetilde{\bfV}_2\Vert \leq\, & \left\Vert \widetilde{\bfU}_2^{\top}\left[\bfY_0 + h\bfF(t_{\sfrac12},\wh\bfY_{\star}) \right]\widetilde{\bfV}_2\right\Vert + \wh c_4h^3 + h^2\varepsilon_{r}\\
\leq\, & h\left\Vert \widetilde{\bfU}_2^{\top}\bfM(t_{\sfrac12},\wh\bfY_{\star})\widetilde{\bfV}_2\right\Vert+ \wh c_4h^3 + h\varepsilon_{2r} + h^2\varepsilon_{r}\\
=\,& \wh c_4h^3 + h\varepsilon_{2r} + h^2\varepsilon_{r}\,,
\end{align*}
where we use that $\widetilde{\bfU}_2^{\top}\wh\bfU_{0} = \wh\bfV_{0}^{\top}\widetilde{\bfV}_2 = 0$.
\end{proof}
Hence, we can conclude that the local error has a robust error bound of order three:
\begin{theorem}
    Under our Assumptions~\ref{as:assumptions}, the local error of the parallel integrator of Section~\ref{sec:v1} fulfills
    \begin{align*}
        \Vert \widetilde\bfA(t_1) - \bfY_1 \Vert \leq \widetilde C_1 h\varepsilon_{r} + \widetilde C_2 h^3 + \vartheta\,,
    \end{align*}
    and the parallel integrator of Section~\ref{sec:v2} fulfills
    \begin{align*}
        \Vert \widetilde\bfA(t_1) - \bfY_1 \Vert \leq \widetilde C_1 h^2\varepsilon_{r} + \widetilde C_2 h\varepsilon_{2r} + \widetilde C_3 h^3 + \vartheta\,,
    \end{align*}
    where constants only depend on the Lipschitz constant and bound of $\bfF$, a bound of third derivatives of the exact solution, and an upper bound of the time stepsize.
\end{theorem}
\begin{proof}
    This result directly follows from the construction of the truncation step, which yields $\Vert \wh\bfY_1 - \bfY_1 \Vert \leq \vartheta$. Then, all terms in \eqref{eq:boundGlobal} are bounded, proving the theorem.
\end{proof}
The global error bound of Theorem~\ref{th:robustglobal} is then directly derived with Lady Windermere's fan  \cite[II.3]{HairerNorsettWanner:ODE_BOOK1} similar to, e.g.,~\cite{KieriLubichWalach,KiV19,CeL22,CeKL22,CeKL23,ceruti2024robust}.

% This should also hold for the original parallel integrator?
%\begin{align*}
%\langle \bfZ, \bfF(t, \bfZ)\rangle = 0 \quad \text{ for all }\bfZ \in\R^{m\times n} \text{ and all } t\in\R^+\, .
%\end{align*}
%\begin{theorem}

%\end{theorem}
%\begin{proof}
%\begin{align*}
%\frac12\frac{d}{dt}\Vert \wh\bfS \Vert^2 = \langle \wh \bfS, \dot{\wh\bfS}  \rangle = \langle \bar \bfS, \wh\bfU_0^{\top} \bfF(\wh\bfY)\wh\bfV_0 - \widetilde\bfU_2^{\top}\bfF(\wh\bfY)\widetilde\bfV_2 \rangle
%\end{align*}
%\end{proof}

\section{Norm preservation up to high order terms}\label{sec:preservation}
The high-order parallel integrators preserve the norm up to high-order terms of the time step size, the truncation tolerance $\vartheta$, and the normal component $\varepsilon_r$, if the original full-rank equation does. This result is similar to the structure--preservation of the augmented BUG integrator \cite[Section~4]{CeKL22}, which preserves the norm up to the truncation tolerance, independent of the time step size. If the right-hand side $\bfF$ fulfills
\begin{align}\label{eq:normpres}
    \langle \bfZ, \bfF(t, \bfZ)\rangle = 0 \quad \text{ for all } \bfZ\in\mathbb{R}^{m\times n} \text{ and all }t,
\end{align}
then the solution of \eqref{eq:exact} preserves the Frobenius norm, that is, $\Vert \bfA(t)\Vert = \Vert \bfA(0) \Vert$ for all $t$. The presented second-order parallel integrators inherit this norm preservation property up to the truncation tolerance $\vartheta$ and an $h^4$ order term.
\begin{theorem}\label{th:normpres}
    If the right-hand side $\bfF$ fulfills \eqref{eq:normpres}, then, under the assumptions in Section~\ref{sec:robusterror}, $\bfY_1$ after a step of the second-order parallel integrator fulfills
    \begin{align*}
        \left\vert \Vert \bfY_1 \Vert - \Vert \bfY_0 \Vert \right\vert \leq \frac{\vartheta^2}{\Vert \bfY_1 \Vert + \Vert \bfY_0 \Vert} + C(h^4+ h^3\varepsilon_r + h^2\varepsilon_r^2)\,,
    \end{align*}
    where the constant $C$ is independent of singular values of the exact or approximate solution.
\end{theorem}

\begin{proof}
    Define $\bfY_S(t) := \bfU_0\bar{\bfS}(t)\bfV_0^{\top}$, $\bfY_L(t) := \bfU_0\bfL(t)^{\top}\widetilde \bfV_2\widetilde \bfV_2^{\top}$, and $\bfY_K(t) :=\widetilde \bfU_2\widetilde \bfU_2^{\top} \bfK(t)\bfV_0^{\top}$. Then, $\wh\bfY(t) := \bfY_S(t) + \bfY_L(t) + \bfY_K(t)$ fulfills $\wh\bfY(t_0) = \bfY_0$ and $\wh\bfY(t_1) = \wh\bfY_1$. Moreover, with the Frobenius inner product~$\langle \cdot, \cdot \rangle$ we have
    \begin{align}\label{eq:consNorm0}
        \frac12 \frac{d}{dt} \Vert \wh\bfY \Vert^2 = \langle\wh\bfY(t), \dot{\wh\bfY}(t) \rangle = \langle \bfY_S(t) + \bfY_L(t) + \bfY_K(t), \dot\bfY_S(t) + \dot\bfY_L(t) + \dot\bfY_K(t) \rangle\,.
    \end{align}
    Note that
    \begin{align*}
        \langle \bfY_S(t), \dot\bfY_L(t)\rangle =\,& \langle\wh\bfU_0\bar{\bfS}(t)\wh\bfV_0^{\top}, \wh\bfU_0\wh\bfU_0^{\top}\bfF(t,\wh\bfU_0\bfL(t)^{\top})\widetilde\bfV_2\widetilde\bfV_2^{\top}\rangle\\
        =\,&\langle\bar{\bfS}(t), \wh\bfU_0^{\top}\bfF(t,\wh\bfU_0\bfL(t)^{\top})\widetilde\bfV_2\widetilde\bfV_2^{\top}\wh\bfV_0\rangle = 0\,.
    \end{align*}
    With the same arguments for remaining cross terms in \eqref{eq:consNorm0}, we directly get
    \begin{align}\label{eq:consNorm1}
        \frac12 \frac{d}{dt} \Vert \wh\bfY \Vert^2 = \langle \bfY_S(t), \dot\bfY_S(t)\rangle + \langle\bfY_L(t),\dot\bfY_L(t)\rangle + \langle\bfY_K(t), \dot\bfY_K(t) \rangle\,.
    \end{align}
    The first term in \eqref{eq:consNorm1} is zero:
    \begin{align*}
        \langle \bfY_S(t), \dot\bfY_S(t)\rangle =\,& \langle \wh\bfU_0\bar{\bfS}(t)\wh\bfV_0^{\top} , \wh\bfU_0\wh\bfU_0^{\top}\bfF(t,\wh\bfU_0\bar{\bfS}(t)\wh\bfV_0^{\top})\wh\bfV_0\wh\bfV_0^{\top}\rangle\\
        =\,& \langle \wh\bfU_0\bar{\bfS}(t)\wh\bfV_0^{\top} , \bfF(t,\wh\bfU_0\bar{\bfS}(t)\wh\bfV_0^{\top})\rangle \stackrel{\eqref{eq:normpres}}{=} 0\,.
    \end{align*}
    The remainder directly follows from the discussion of Section~\ref{sec:robusterror}: The absolute value of the third term in \eqref{eq:consNorm1} is bounded according to
    \begin{align}\label{eq:scProdCons}
        |\langle\bfY_K(t),\dot\bfY_K(t)\rangle| =\,& |\langle\widetilde\bfU_2\widetilde\bfU_2^{\top} \bfK(t)\wh\bfV_0^{\top}, \widetilde\bfU_2\widetilde\bfU_2^{\top}\bfF(t,\bfK(t)\wh\bfV_0^{\top})\wh\bfV_0\wh\bfV_0^{\top}\rangle| \nonumber\\
        \leq\,&\Vert\widetilde\bfU_2\widetilde\bfU_2^{\top}\bfK(t)\Vert\cdot\Vert \widetilde\bfU_2\widetilde\bfU_2^{\top}\bfF(t,\bfK(t)\wh\bfV_0^{\top})\wh\bfV_0\Vert.
    \end{align}
    From the construction of $\widetilde\bfU_2$ we know that $\widetilde\bfU_2\widetilde\bfU_2^{\top}[\bfK(t_0), \bfF(t_0, \bfY_0)\bfV_0] = 0$. Therefore, 
    \begin{align*}
        \Vert\widetilde\bfU_2\widetilde\bfU_2^{\top}\bfK(t)\Vert \leq\,& \Vert\widetilde\bfU_2\widetilde\bfU_2^{\top} \left[\bfK(t_0) + h\bfF(t_0, \bfY_0)\wh\bfV_0 \right]\Vert + c_1h^2\, \\
        =\,& h\Vert\widetilde\bfU_2\widetilde\bfU_2^{\top} \bfM(t_0, \bfY_0)\wh\bfV_0\Vert + c_1h^2 + h\varepsilon_r =  c_1h^2 + h\varepsilon_r\,, 
    \end{align*}
    where we use that $\widetilde\bfU_2\widetilde\bfU_2^{\top} \bfM(t_0, \bfY_0)\wh\bfV_0 = \widetilde\bfU_2\widetilde\bfU_2^{\top} \bfF(t_0, \bfY_0)\bfV_0 = 0$. For the second term in \eqref{eq:scProdCons}, we get
    \begin{align*}
        \Vert\widetilde\bfU_2\widetilde\bfU_2^{\top}\bfF(t,\bfK(t)\wh\bfV_0^{\top})\wh\bfV_0\Vert \leq\,& \Vert\widetilde\bfU_2\widetilde\bfU_2^{\top} \bfF(t_0, \bfY_0)\wh\bfV_0\Vert + c_2h\\
        \leq\,& \Vert\widetilde\bfU_2\widetilde\bfU_2^{\top} \bfM(t_0, \bfY_0)\wh\bfV_0\Vert + c_2h + \varepsilon_r\\
        =\,&c_2h + \varepsilon_r\,,
    \end{align*}
    with constants $c_1$ and $c_2$ independent of small singular values of the solution. Hence, $|\langle\bfY_K(t),\dot\bfY_K(t)\rangle| \lesssim h^3 + h^2\varepsilon_r + h\varepsilon_r^2$ and with an analog derivation, we have 
    \begin{align*}
        |\langle\bfY_L(t),\dot\bfY_L(t)\rangle| \lesssim h^3+ h^2\varepsilon_r + h\varepsilon_r^2\,.
    \end{align*}
    Therefore, we have for
    \begin{align*}
        \big\vert \Vert \wh\bfY_1 \Vert^2 - \Vert \bfY_0 \Vert^2 \big\vert = \left\vert \int_{t_0}^{t_1}\frac{d}{dt} \Vert \wh\bfY(t) \Vert^2\,dt \right\vert \leq c_3(h^4+ h^3\varepsilon_r + h^2\varepsilon_r^2)\,.
    \end{align*}
    Moreover, we know from the truncation step with $\bfSigma = \diag(\sigma_j)_{j=1}^{r_1}$ and $\wh{r} \in\{3r, 4r\}$ that
    \begin{align*}
        \Vert \wh\bfY_1 \Vert^2 - \Vert \bfY_1 \Vert^2 = \Vert \wh\bfSigma \Vert^2 - \Vert \bfSigma \Vert^2 = \sum_{j=r_1+1}^{\wh r} \sigma_j^2 \le \vartheta^2\,
    \end{align*}
    and therefore 
    \begin{align*}
        \big\vert \Vert \bfY_1 \Vert^2 - \Vert \bfY_0 \Vert^2 \big\vert =\,& \big\vert \Vert \bfY_1 \Vert^2 - \Vert \wh\bfY_1 \Vert^2 + \Vert \wh\bfY_1 \Vert^2 - \Vert \bfY_0 \Vert^2 \big\vert\\
        \leq\,& \vartheta^2 + c_3(h^4+ h^3\varepsilon_r + h^2\varepsilon_r^2)\,.
    \end{align*}
    With $\left\vert \Vert \bfY_1 \Vert - \Vert \bfY_0 \Vert \right\vert \cdot (\Vert \bfY_1 \Vert + \Vert \bfY_0 \Vert) = \big\vert \Vert \bfY_1 \Vert^2 - \Vert \bfY_0 \Vert^2 \big\vert$ we can then prove the theorem.
\end{proof}

\begin{remark}
    A similar derivation for the first-order parallel integrator yields a second-order local error. Denoting the solution obtained with the first-order parallel integrator as $\bfY_1^{(1)}$ we have
    \begin{align*}
        \left\vert \Vert \bfY_1^{(1)} \Vert - \Vert \bfY_0 \Vert \right\vert \leq \frac{\vartheta^2}{\Vert \bfY_1^{(1)} \Vert + \Vert \bfY_0 \Vert} + Ch^2\,.
    \end{align*}
    The proof uses $\widetilde \bfU_1$ and $\widetilde \bfV_1$ instead of $\widetilde \bfU_2$ $\widetilde \bfV_2$, while making use of
    \begin{align*}
        \Vert\widetilde\bfU_1\widetilde\bfU_1^{\top}\bfK(t)\Vert \leq c_1h\,, \text{ and }\;\Vert\widetilde\bfU_1\widetilde\bfU_1^{\top}\bfF(t_0,\bfK(t)\bfV_0^{\top})\bfV_0\Vert = O(1)\,.
    \end{align*}
\end{remark}

\begin{remark}
    A similar result is achieved by directly investigating the squared norm $\langle \wh\bfY, \wh\bfY \rangle$, which leads to bounding 
    \begin{align*}
        |\langle\bfY_K(t),\bfY_K(t)\rangle|
        \leq\,&\Vert\widetilde\bfU_2\widetilde\bfU_2^{\top}\bfK(t)\Vert^2.
    \end{align*}
\end{remark}

\begin{remark}
    The result of Theorem~\ref{th:normpres} is a clear advantage compared to the first-order parallel integrator. However, the augmented and midpoint BUG integrators enable preservation independent of the time step size. Note, however, that for most kinetic problems, structure preservation can be enforced by other means via further basis augmentations~\cite{Einkemmer2023,Einkemmer2023a}, or micro-macro decompositions \cite{koellermeier2023macro}.
\end{remark}

%\subsection{Gradient systems}

%\subsection{Gradient systems}
%Given a function $f:\mathbb{R}^{m\times n}\rightarrow \mathbb{R}$ and the gradient field $\bfG(\bfA) = \nabla f(\bfA)$, along every path $\bfA(t)$ we have
%\begin{align*}
%    \frac{d}{dt} f(\bfA(t)) = \langle \bfG(\bfA(t)), \dot \bfA(t)\rangle\,.
%\end{align*}
%Then, if $\bfA(t)$ fulfills $\dot \bfA(t) = -\bfG(\bfA(t))$, the value of $f(\bfA(t))$ decreases over time.
%
%\begin{proof}
%    Defining $\wh\bfY(t)$ as in the proof of Theorem~\ref{th:normpres}, we have
%    \begin{align}\label{eq:gradsys0}
%        \frac{d}{dt} f( \wh\bfY(t) ) = \langle\bfG(\wh\bfY(t)), \dot{\wh\bfY}(t) \rangle = \langle \bfG(\wh\bfY(t)), \dot\bfY_S(t) + \dot\bfY_L(t) + \dot\bfY_K(t) \rangle\,.
%    \end{align}
%\end{proof}

\section{Numerical results}\label{sec:numExp}
In this section, we present the results of different numerical experiments.
These numerical simulations are implemented using \textsc{Julia} 1.9.4. All numerical results can be reproduced with the openly available source code \cite{code}.

\subsection{Discrete Schr\"odinger equation}
To investigate the properties of the parallel integrators numerically, we study the discrete Schr\"odinger equation
\begin{equation}%\mathrm{i}
	   i\dot{\bfY}(t) = \bfH[\bfY(t)], \quad \bfY(t_0) = \bfU_0 \bfS_0 \bfV_0^\top \in \R^{n \times n},
	\label{eq:discr_schoedinger}
\end{equation}
where with the $j$-th unit vector $\mathbf{e}_j$ we have
\begin{align*}
	&\bfH[\bfY] = -\frac{1}{2} \Big(\bfD \bfY + \bfY \bfD^\top \Big) + \bfV_\text{cos}\bfY\bfV_\text{cos} \in \R^{n \times n},
	\\
	& \bfD = \texttt{tridiag}(-1,2,-1) + \mathbf{e}_1\mathbf{e}_n^{\top} + \mathbf{e}_n\mathbf{e}_1^{\top} \in \R^{n \times n} , 
	\\
	& \bfV_\text{cos} = \text{diag}(1-\cos(2\pi j / n)) \quad j = -n/2, \cdots, n/2 - 1 . 
\end{align*}
As initial condition, we choose random orthonormal basis matrices $\bfU_0, \bfV_0 \in\mathbb{R}^{n\times n}$ as well as $\bfS_0 = \text{diag}(10^{-i})_{i = 1}^n$. The initial condition is then truncated to the rank chosen for the individual integrators.

The reference solution is computed with the \textsc{Julia} solver \texttt{ode45} and tolerance parameters \textsc{\{'RelTol', 1e-10, 'AbsTol', 1e-10\} }. We use the same solver to solve all matrix ODEs arising in the DLRA integrators. Numerical results are compared with the first-order parallel integrator \cite{CeKL23} and the second-order midpoint BUG integrator~\cite{ceruti2024robust}. The number of spatial grid points per dimension is $n=100$; the chosen ranks are $r\in\{5,10,15\}$, and the final time is $T = 1$.
\begin{figure}[ht]
    \centering
    \includegraphics[width=\textwidth]{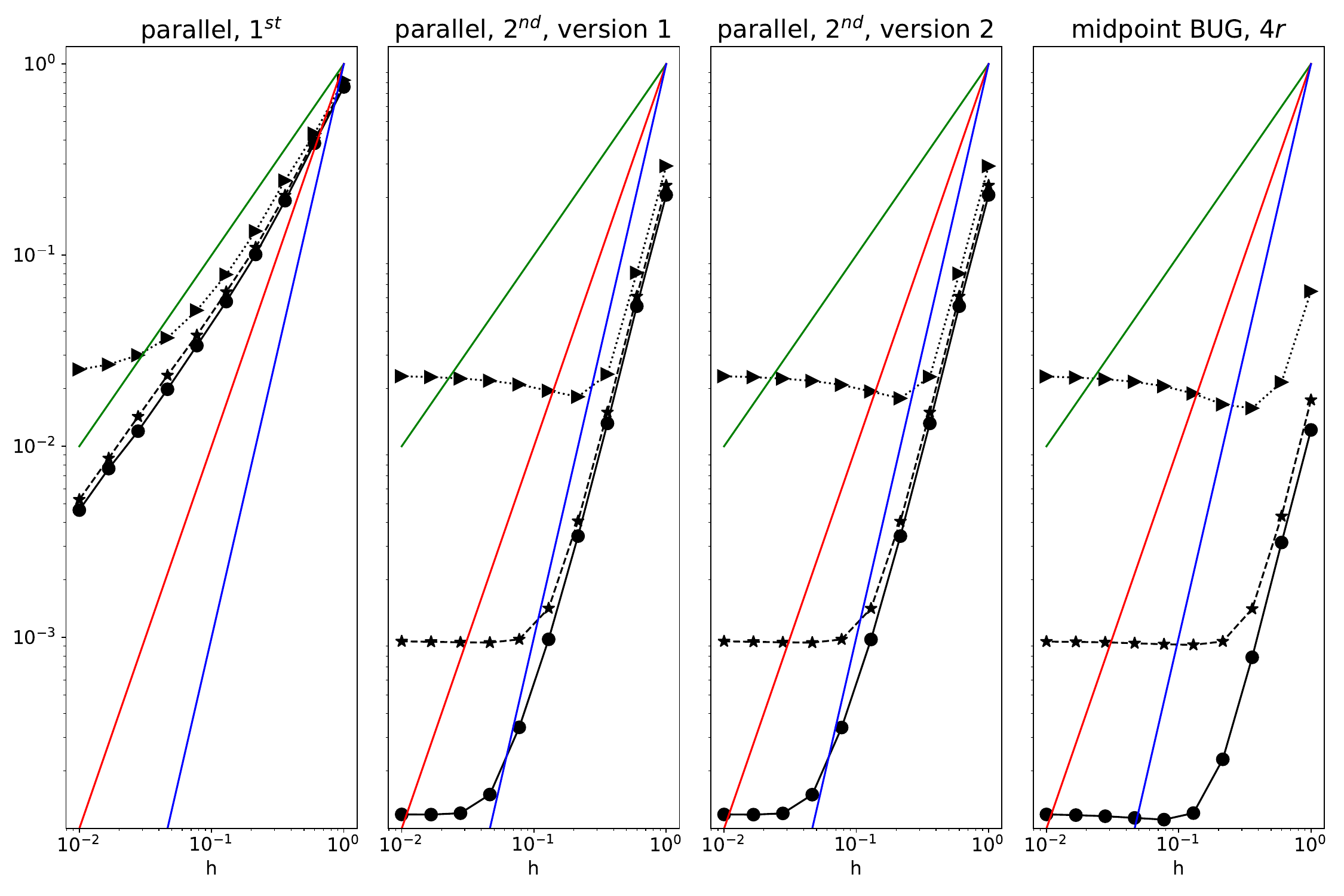}
    \caption{Comparison of relative approximation errors measured in Frobenius norm for different integrators for ranks $r\in\{5,10,15\}$. Solid lines indicate convergence slopes of orders one, two, and three.}
    \label{fig:nonStiff_T1_err}
\end{figure}
Figure~\ref{fig:nonStiff_T1_err} shows the error of the different integrators for various timestep sizes $h$. It is observed that while the parallel integrator yields large errors and shows a first-order error bound, the second--order parallel integrators lead to significantly improved errors while converging with order two. Both variants of the parallel integrator show similar error behavior. It has to be noted that the integrator of \cite{ceruti2024robust} shows a significantly improved error, while however, requiring three sequential solves of matrix ODEs with a maximal rank of $4r$. Recall that the parallel integrators of this work solve all differential equations in one parallel step while requiring a maximal rank of $2r$. The error of the norm is depicted in Figure~\ref{fig:nonStiff_T1_norm}. We observe that the second-order parallel integrators' norm preservation is significantly improved compared to the first-order parallel integrator and shows an additional order compared to the results in Section~\ref{sec:preservation}. However, the midpoint BUG integrator yields a significantly improved norm preservation property similar to the overall error. 
\begin{figure}[ht]
    \centering
    \includegraphics[width=\textwidth]{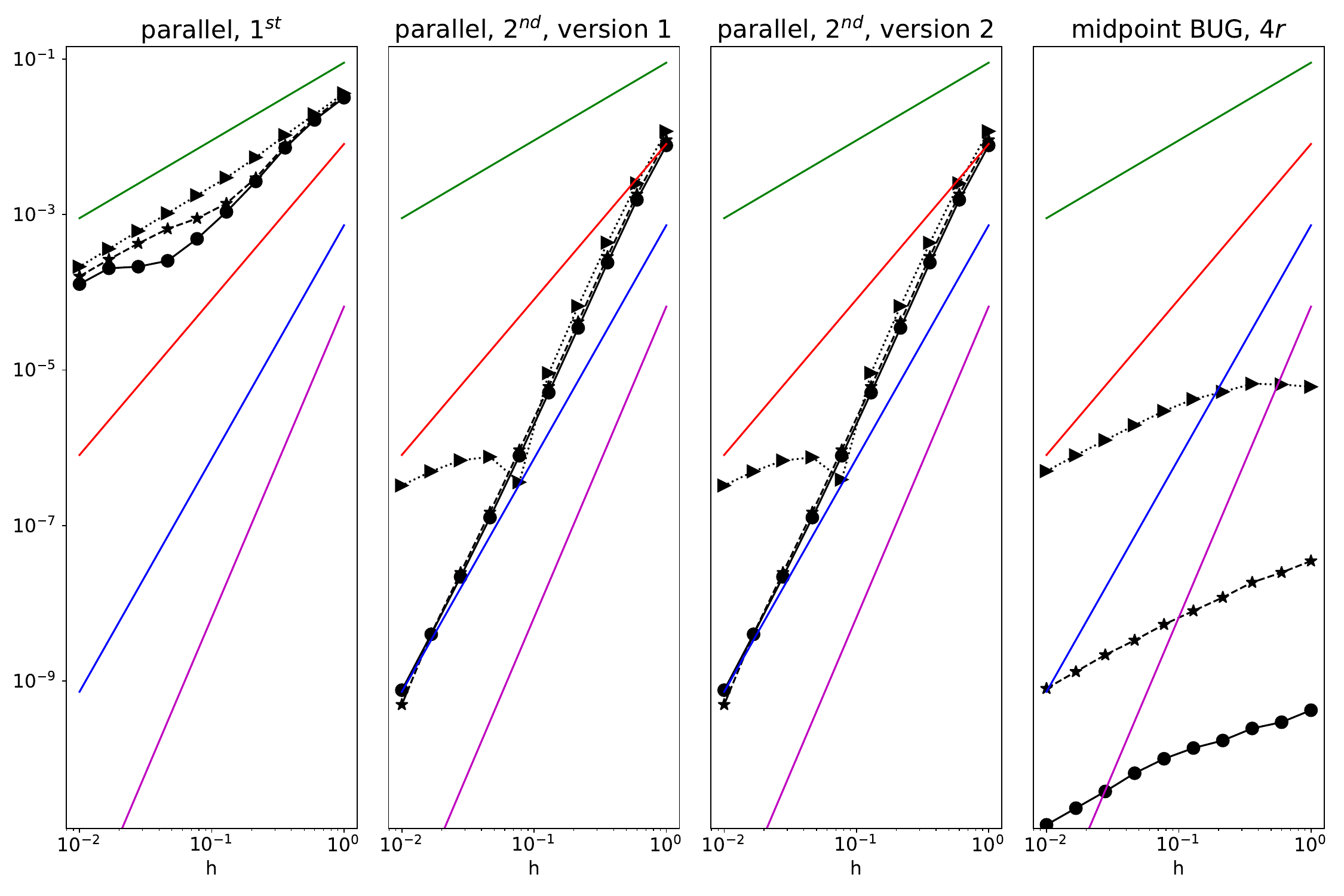}
    \caption{Comparison of Frobenius norm error for different integrators for ranks $r\in\{5,10,15\}$. Solid lines indicate convergence slopes of orders one, two, three, and four.}
    \label{fig:nonStiff_T1_norm}
\end{figure}

\subsection{Lattice testcase of radiative transfer}
To demonstrate the computational costs of the proposed method, we study the lattice testcase, which mimics the blocks in a nuclear reactor. The two-dimensional radiative transfer equation describes the underlying dynamics, which model the interaction of radiation particles with a background material. 
\begin{equation}
	\label{eq:rtela}
	\begin{aligned}
		&\partial_t f + \mathbf{\Omega}\cdot\nabla f + \sigma_t(\mathbf{x})f = \frac{\sigma_s(\mathbf{x})}{4\pi} \int_{\mathbb{S}^2} f \,d\mathbf{\Omega}+Q(\mathbf{x}),\qquad 
		(\mathbf{x}, \mathbf{\Omega}) \in [0,7]^2 \times \mathbb{S}^2 \;. \\[1mm]
		& f(t = 0) = 10^{-9}.
	\end{aligned}
\end{equation}
The solution is the radiation density or angular flux $f = f(t,\bfx,\bfOmega)$, where $t\in\mathbb{R}_+$ is time, $\bfx\in [-1,5, 1.5 ]\times [-1,5, 1.5 ]$ is the spatial position and $\bfOmega\in\mathbb{S}^2$ is the direction of flight, that is, the direction in which particles travel. The interaction with the background material is modeled through interaction cross-sections $\sigma_a(\mathbf{x})$ and $\sigma_s(\mathbf{x})$ such that $\sigma_t(\mathbf{x}) = \sigma_a(\mathbf{x}) + \sigma_s(\mathbf{x})$. Here, $\sigma_a$ models absorption, and $\sigma_s$ models scattering, i.e., collisions of particles with the background material. Moreover, the material includes an isotropic source $Q$. The positioning of these cross-sections is illustrated in Figure~\ref{fig:setup_lattice_testcase}, where white blocks represent a purely scattering material with $\sigma_s = 1$, blue blocks are purely absorbing materials with $\sigma_a = 10$. The red block consists of a purely absorbing material ($\sigma_a = 10$) with a source of $Q=1$. A reference solution computed with the spherical harmonics (P$_N$) method can be found in Figure~\ref{fig:scalar_flux_PN_Lattice_nx250_N21}. The reference solution is calculated with a polynomial degree of $21$ ($441$ expansion coefficients) in the direction of flight $\mathbf{\Omega}$ and $350^2 = 122500$ spatial cells. Here, the scalar flux $\Phi(T,\mathbf x) = \int_{\mathbb{S}^2} f(T,\mathbf x,\mathbf{\Omega}) \,\mathrm{d}\mathbf{\Omega}$ is plotted using a logarithmic scale as commonly done for this benchmark.
\begin{figure}
     \centering
     \begin{subfigure}[b]{0.49\textwidth}
         \centering
         \includegraphics[height=0.9\textwidth]{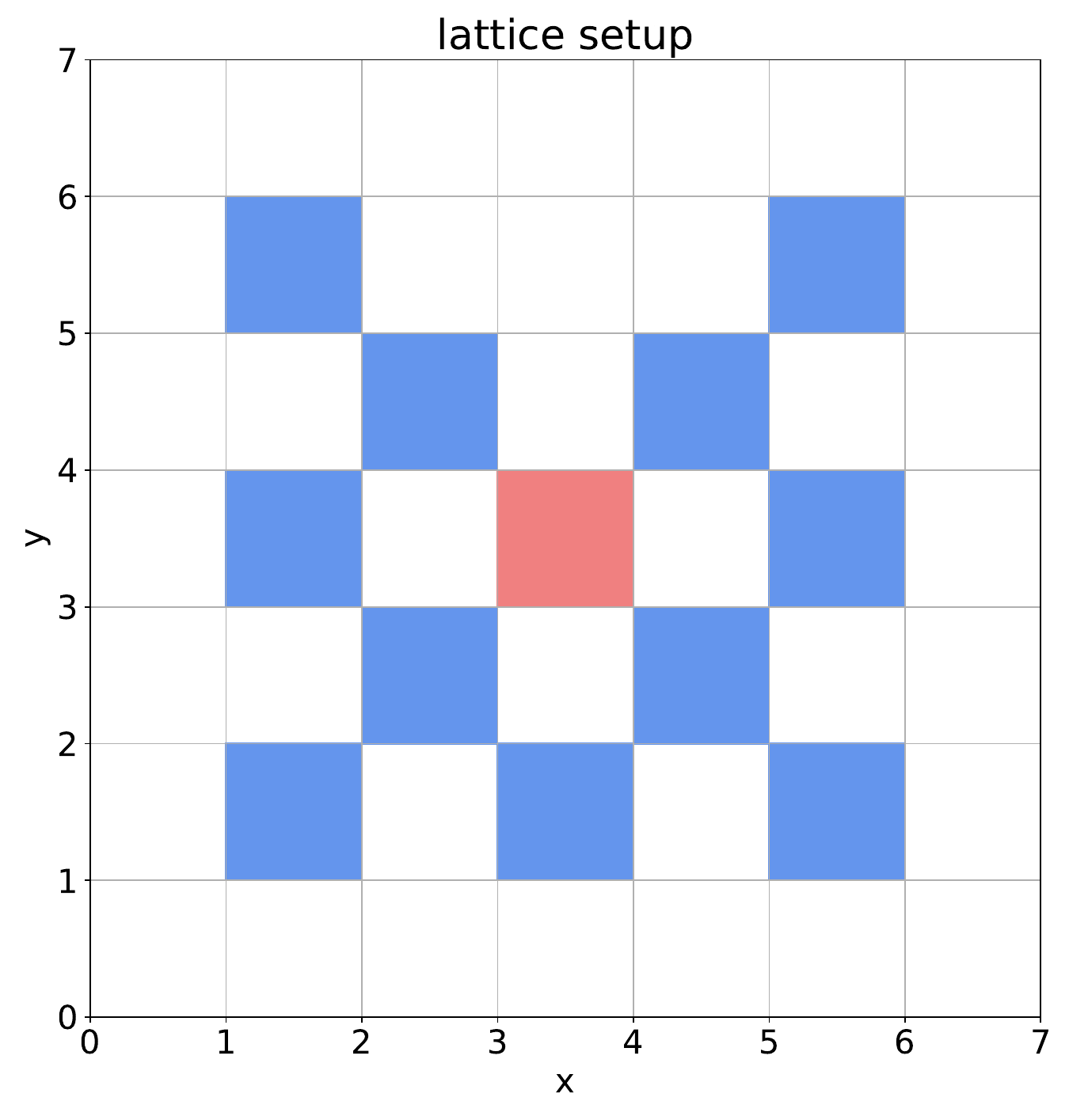}
         \caption{}\label{fig:setup_lattice_testcase}
     \end{subfigure}
     \hfill
     \begin{subfigure}[b]{0.49\textwidth}
         \centering
         \includegraphics[height=0.9\textwidth]{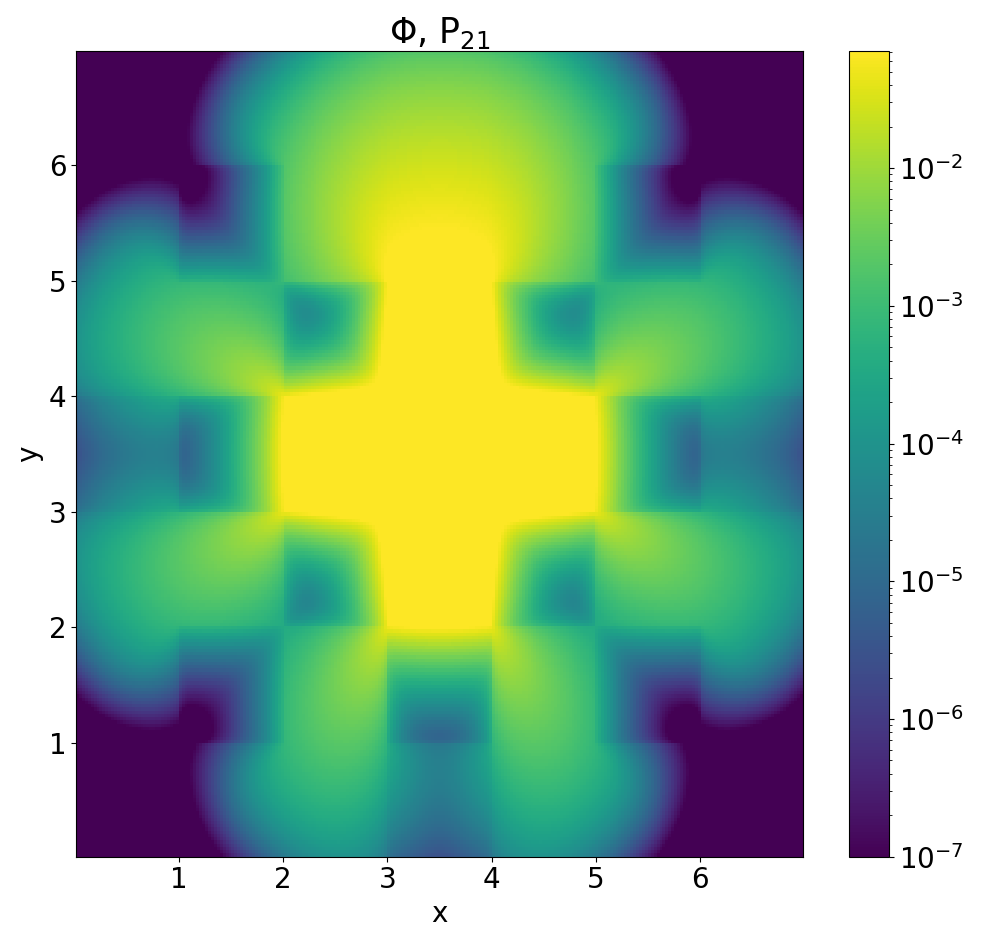}
         \caption{}
         \label{fig:scalar_flux_PN_Lattice_nx350_N21}
     \end{subfigure}
            \caption{Left: Spatial setup of the lattice. Right: Reference solution at time $T=3.2$ in logarithmic scale computed with the full-rank spherical harmonics (P$_N$) method using spherical harmonics up to degree $N = 21$.}
	\label{fig:scalar_flux_PN_Lattice_nx250_N21}
\end{figure}

To reduce computational costs and memory footprint, we wish to approximate the solution to the lattice testcase with dynamical low-rank approximation. We compare the second version of the new second-order parallel integrator with the rank $4r$ midpoint BUG integrator since these two integrators share similar dependence on normal components. The rank is fixed at $r = 10$ and $r = 20$. Hence, in the truncation step, we do not truncate according to a tolerance parameter but instead to a fixed rank. All substeps are solved with Heun's method, and we refer to \cite[Section~5]{CeKL23} for further information on the numerical discretization. The corresponding results for rank $r = 10$ can be found in Figure~\ref{fig:latticeRank10}, and the results for rank $r=20$ are depicted in Figure~\ref{fig:latticeRank20}. As is common for the lattice benchmark, the scalar flux $\Phi(T, \bfx)$ is depicted in a logarithmic scale, where negative function values are shown in white. It is observed that both integrators yield similar solutions. Note that the generation of negative function values is common for modal methods. The computational time of the second--order parallel integrator is substantially improved, with a runtime of 174 seconds for rank $20$ and 78 seconds for rank $10$ compared to the runtime of the midpoint BUG integrator, which is 265 seconds for rank $20$ and  132 for rank $10$. Note that neither code uses parallelism, and the reduced computational costs result solely from the matrix ODEs' reduced rank.

\begin{figure}[ht]
    \centering
    \includegraphics[height=0.49\textwidth]{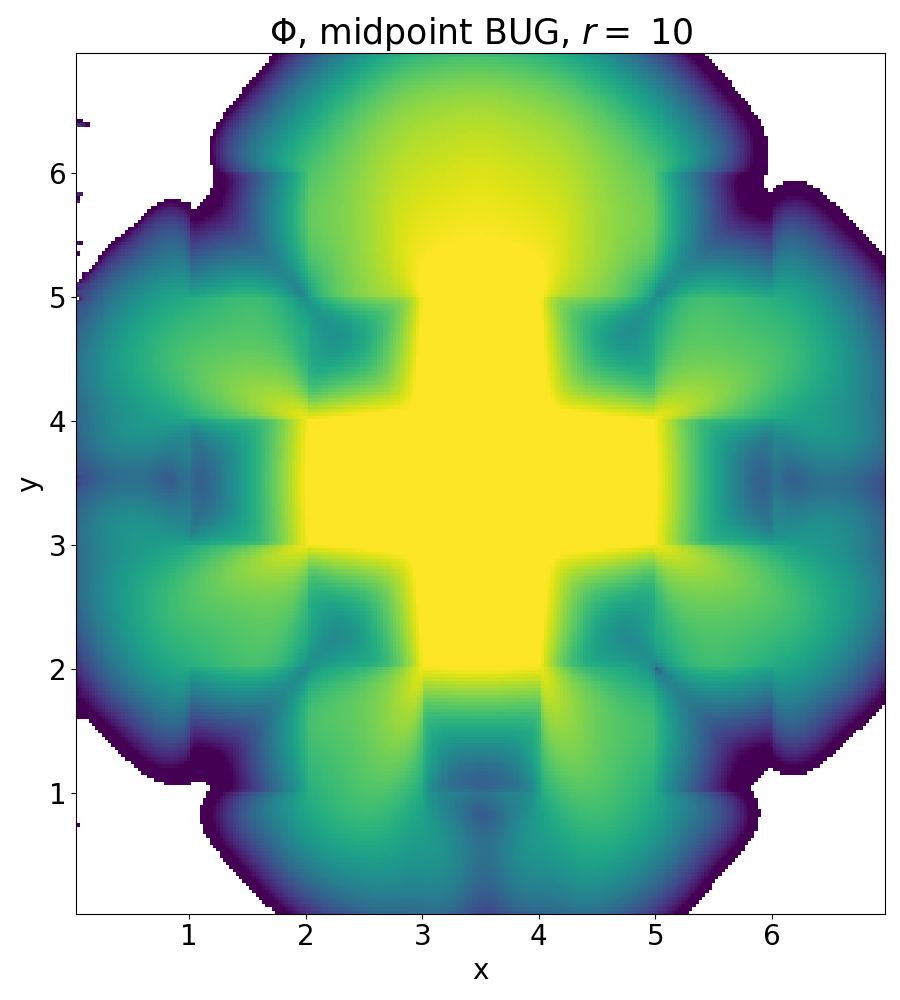}
    \includegraphics[height=0.49\textwidth]{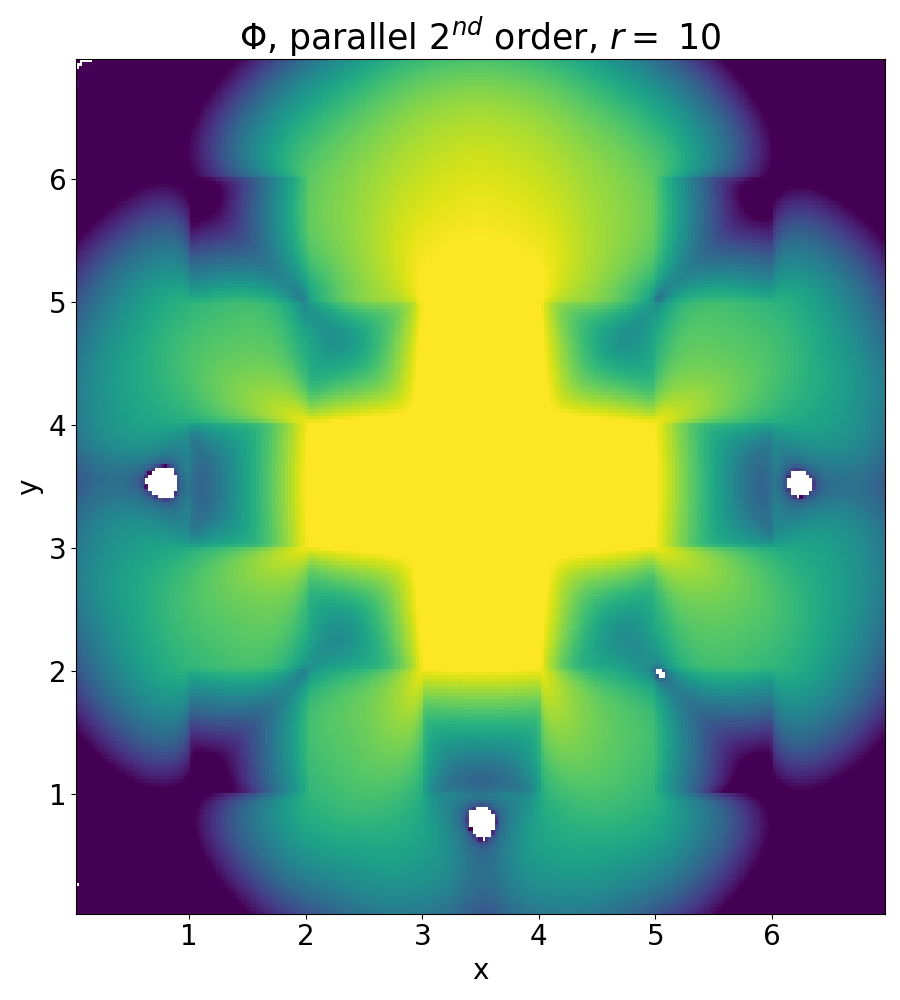}
    \caption{Left: Scalar flux of the midpoint BUG integrator. Left: Scalar flux of version~2 of the second-order parallel BUG integrator. Both solutions use a fixed rank $r=10$ and are plotted at time $T=3.2$.}
    \label{fig:latticeRank10}
\end{figure}

\begin{figure}[ht]
    \centering
    \includegraphics[height=0.49\textwidth]{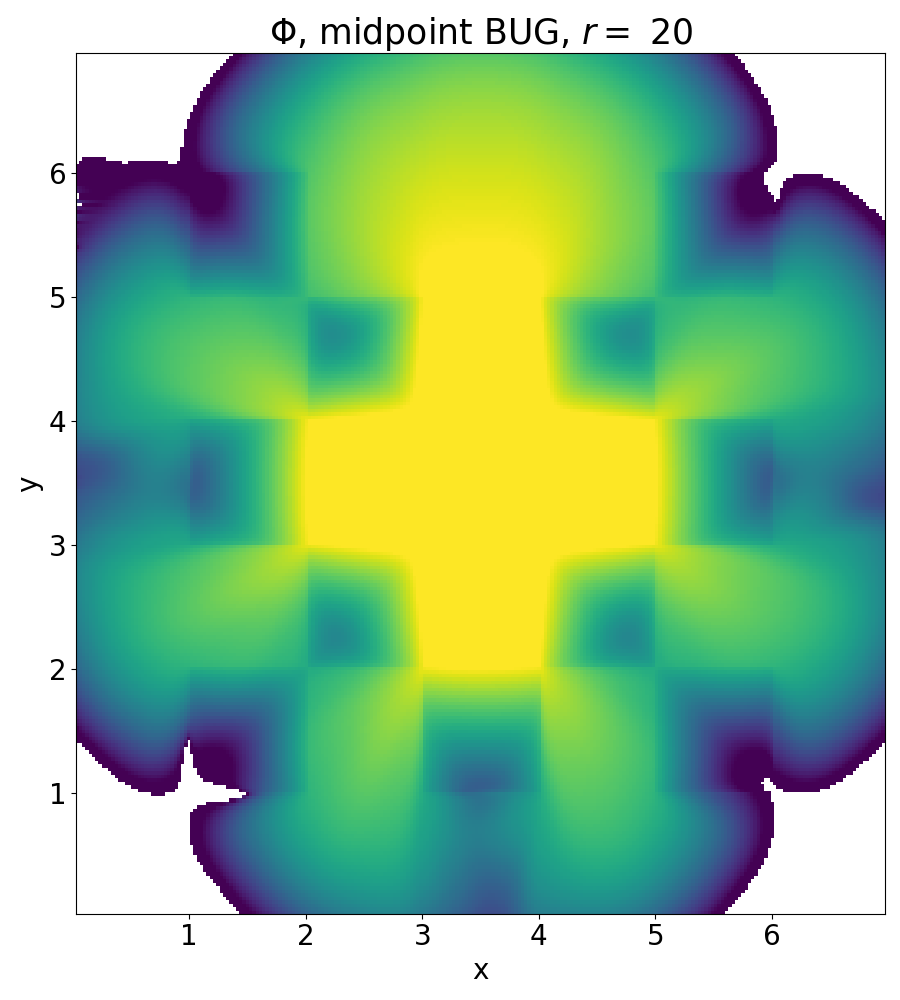}
    \includegraphics[height=0.49\textwidth]{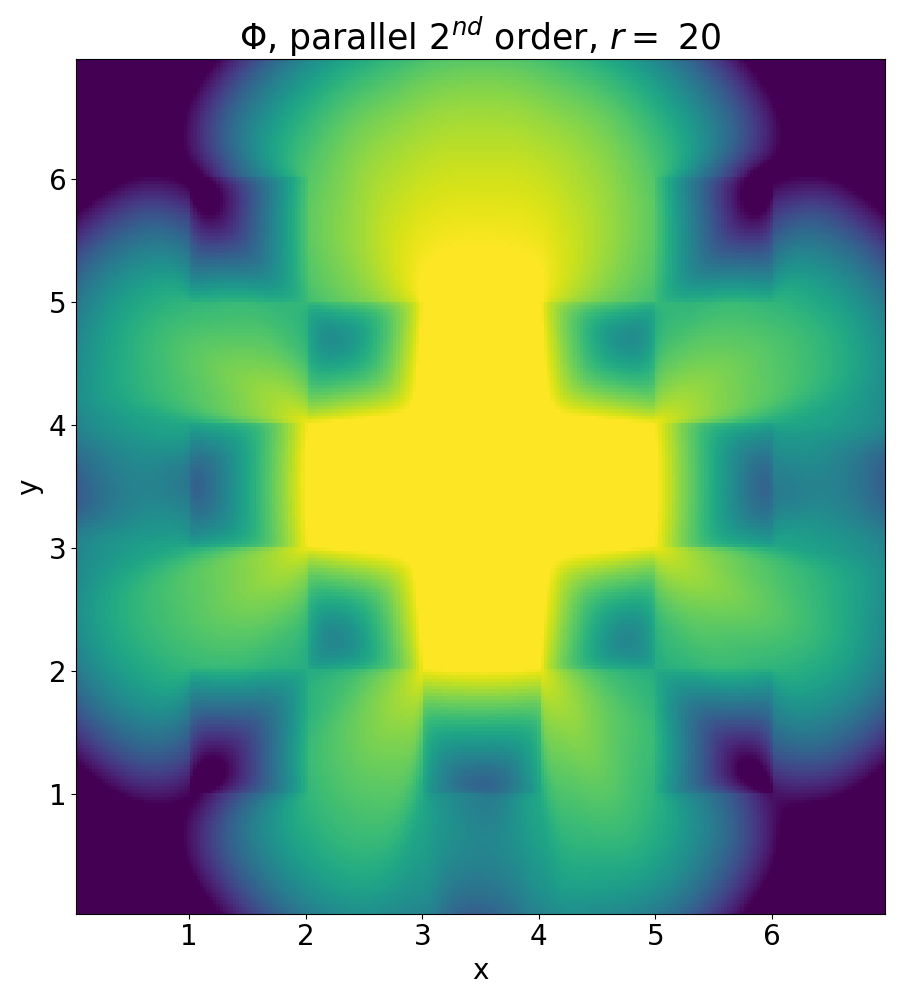}
    \caption{Left: Scalar flux of the midpoint BUG integrator. Left: Scalar flux of version~2 of the second-order parallel BUG integrator. Both solutions use a fixed rank $r=20$ and are plotted at time $T=3.2$.}
    \label{fig:latticeRank20}
\end{figure}

\section{Conclusion and outlook}\label{sec:conclusionoutlook}
In this work, two novel second--order parallel BUG integrators have been proposed. Moreover, a rigorous analysis to a second--order robust error bound has been presented, which for the integrator of Section~\ref{sec:v2} has led to an improved dependence on normal components, similar to the $4r$ midpoint BUG integrator of \cite{ceruti2024robust}. In addition, an enhanced error in conserving the norm has been presented. Numerical experiments show the derived error behavior, which is substantially reduced compared to the first-order parallel BUG integrator. While the resulting error is more significant than that obtained with the midpoint BUG integrator of \cite{ceruti2024robust}, the parallel integrator yields significantly reduced computational times, even when not using its inherent parallelism.

While parallelism is limited in the matrix case, the integrator is expected to extend to tree-tensor networks, similar to \cite{ceruti2023rank}, where the parallelism of the integrator can significantly reduce computational times. Moreover, the integrator is expected to allow straightforward extensions to preserve solution invariants and allow for asymptotic--preserving schemes for kinetic problems, similar to \cite{einkemmer2023conservation,patwardhan2024asymptotic}.

\section*{Acknowledgments}
I want to thank Christian Lubich and Dominik Sulz for many helpful discussions during my stay in T\"ubingen.

\bibliographystyle{siamplain}
\bibliography{main}
\end{document}

%% file: ex_shared.tex
\usepackage{lipsum}
\usepackage{amsfonts}
\usepackage{graphicx}
\usepackage{epstopdf}
\usepackage{algorithmic}
\usepackage{amssymb}
\usepackage{subcaption}
\ifpdf
  \DeclareGraphicsExtensions{.eps,.pdf,.png,.jpg}
\else
  \DeclareGraphicsExtensions{.eps}
\fi

\usepackage{xfrac}
\newsiamthm{assumption}{Assumptions}

\def\R{{\mathbb R}}

\def\eps{\varepsilon}

\def\wh{\widehat}

\newcommand\bfx{{\mathbf x}}

\newcommand\bfI{{\mathbf I}}
\newcommand\bfA{{\mathbf A}}
\newcommand\bfB{{\mathbf B}}
\newcommand\bfC{{\mathbf C}}
\newcommand\bfD{{\mathbf D}}

\newcommand\bfF{{\mathbf F}}

\newcommand\bfH{{\mathbf H}}
\newcommand\bfK{{\mathbf K}}
\newcommand\bfL{{\mathbf L}}
\newcommand\bfM{{\mathbf M}}

\newcommand\bfP{{\mathbf P}}
\newcommand\bfQ{{\mathbf Q}}
\newcommand\bfR{{\mathbf R}}
\newcommand\bfS{{\mathbf S}}

\newcommand\bfU{{\mathbf U}}
\newcommand\bfV{{\mathbf V}}

\newcommand\bfX{{\mathbf X}}
\newcommand\bfY{{\mathbf Y}}
\newcommand\bfZ{{\mathbf Z}}

\newcommand\bfSigma{{\mathbf \Sigma}}

\newcommand\bfOmega{{\mathbf \Omega}}

\def\eps{\varepsilon}

\def\phi{\varphi}

\newsiamremark{remark}{Remark}
\newsiamremark{hypothesis}{Hypothesis}
\crefname{hypothesis}{Hypothesis}{Hypotheses}
\newsiamthm{claim}{Claim}

\headers{Second-order robust parallel integrators for DLRA}{J. Kusch}

\title{Second-order robust parallel integrators for dynamical low-rank approximation}

\author{Jonas Kusch\thanks{Norwegian University of Life Sciences, Scientific Computing, \r{A}s, Norway
  (\email{jonas.kusch@nmbu.no}).}}

\usepackage{amsopn}
\DeclareMathOperator{\diag}{diag}